\documentclass[11pt]{amsart}

\usepackage[all]{xy}

\usepackage[utf8]{inputenc}
\usepackage{amscd,amssymb,latexsym,subfigure,verbatim,epsfig}
\usepackage{amsthm}
\usepackage{amsmath}
\usepackage{supertabular}
\usepackage{calc}
\usepackage{listings}

\usepackage{tikz}
\usetikzlibrary{patterns,arrows}
\usepackage[labelfont=bf, font={small}]{caption}[2005/07/16]
\usepackage{pgfplots}
\pgfplotsset{compat=1.8}

\usepackage{multicol} 
\usepackage{enumitem}
\usepackage{xcolor}
\usepackage{thmtools}

\usepackage[top = 3cm, bottom = 3cm, left = 3cm, right=3cm]{geometry}
\usepackage[colorlinks]{hyperref}
\hypersetup{
  linkcolor=[rgb]{0.3,0.3,0.6},
  citecolor=[rgb]{0.2, 0.6, 0.2},
  urlcolor=[rgb]{0.6, 0.2, 0.2}
}
\makeindex

\newtheorem{thm}{Theorem}[section]
\newtheorem{defn}[thm]{Definition}
\newtheorem{prop}[thm]{Proposition}

\newtheorem{lem}[thm]{Lemma}
\newtheorem{cor}[thm]{Corollary}

\theoremstyle{definition}

\newtheorem{examplex}[thm]{Example}
\newenvironment{example}
  {\pushQED{\qed}\examplex}
  {\popQED\endexamplex}

\newcommand{\eps}{\varepsilon}	
\newcommand{\CF}{\mathit{CF}}

\newcommand{\Z}{{\mathbb Z}}

\newcommand{\Tor}{\mathop{\rm Tor}\nolimits}

\newcommand{\Hom}{\mathop{\rm Hom}\nolimits}
\newcommand{\im}{\mathop{\rm im}\nolimits}

\renewcommand{\P}{{\mathbb{P}}}

\renewcommand{\k}{\mathbb{K}}

\usepackage{psfrag}

\newcommand{\reg}{\mathrm{reg}}

\DeclareMathOperator{\End}{End}

\DeclareMathOperator{\Ind}{Ind}

\newcommand{\init}{\mathit{in}}
\newcommand{\full}{\mathit{full}}
\DeclareMathOperator{\pdim}{pdim}

\newcommand{\Sym}{\operatorname{Sym}}
\newcommand{\bR}{\mathbf{R}}
\newcommand{\smcdot}{{\textup{$\cdot$}}}

\newcommand{\rmH}{\mathrm{H}}
\newcommand{\bbC}{\mathbb{C}}
\newcommand{\bbZ}{\mathbb{Z}}
\newcommand{\bbD}{\mathbb{D}}
\newcommand{\bbF}{\mathbb{F}}
\newcommand{\bbK}{\mathbb{K}}
\newcommand{\bfe}{\mathbf{e}}
\newcommand{\bfR}{\mathbf{R}}

\newcommand{\vvirg}{, \ldots ,}
\newcommand{\frakS}{\mathfrak{S}}

\newcommand{\HS}{\mathrm{HS}}
\newcommand{\perm}{\mathrm{perm}}
\newcommand{\bfVBP}{\mathbf{VBP}}
\newcommand{\bfVNP}{\mathbf{VNP}}
\renewcommand{\bar}[1]{\overline{#1}}
\renewcommand{\tilde}[1]{\widetilde{#1}}
\newcommand{\xto}[1]{\xrightarrow{#1}}

\newcommand{\textbinom}[2]{{\textstyle \binom{#1}{#2}}}
\begin{document}

\title[Resolving the 2\texttimes 2 permanents]
{Bernstein-Gelfand-Gelfand meets geometric complexity theory: resolving the 2\texttimes 2 permanents of a 2\texttimes $\mathbf{n}$ matrix}

\author{Fulvio Gesmundo}
\address[F. Gesmundo]{Institut de Mathématiques de Toulouse; UMR5219 -- Université de Toulouse; CNRS -- UPS, F-31062 Toulouse Cedex 9, France}
\email{fgesmund@math.univ-toulouse.fr}

\author{Hang (Amy) Huang}
\address[A. Huang]{Department of Mathematics, Auburn University}
\email{hzh0105@auburn.edu}

\author{Hal Schenck}
\thanks{Schenck supported by NSF DMS-2006410}
\address[H. Schenck]{Department of Mathematics, Auburn University}
\email{hks0015@auburn.edu}

\author{Jerzy Weyman}
\thanks{Weyman supported by MAESTRO NCN - UMO-2019/34/A/ST1/00263 - Research in Commutative Algebra and Representation Theory, NAWA POWROTY - PPN/PPO/2018/1/00013/U/00001 - Applications of Lie algebras to Commutative Algebra, and OPUS grant National Science Centre, Poland grant UMO-2018/29/BST1/01290}
\address[J. Weyman]{Department of Mathematics, Jagiellonian University}
\email{jerzy.weyman@uj.edu.pl}

\makeatletter
\@namedef{subjclassname@2020}{%
  \textup{2020} Mathematics Subject Classification}
\makeatother


\begin{abstract}
We describe the minimal free resolution of the ideal of $2 \times 2$ subpermanents of a $2 \times n$ generic matrix $M$. In contrast to the case of $2 \times 2$ determinants, the $2 \times 2$ permanents define an ideal which is neither prime nor Cohen-Macaulay. We combine work of Laubenbacher-Swanson \cite{LS} on the Gr\"obner basis of an ideal of $2 \times 2$ permanents of a generic matrix with our previous work in \cite{elsw} connecting the initial ideal of $2 \times 2$ permanents to a simplicial complex. The main technical tool is a spectral sequence arising from the Bernstein-Gelfand-Gelfand correspondence. 
\end{abstract}
\maketitle

\section{Introduction}\label{sec:one}
For a matrix $M$ whose entries are the variables in a polynomial ring, there are
two classes of ideals which arise naturally: the ideal consisting of all $k \times k$ 
subdeterminants of $M$, and the ideal consisting of all $k \times k$ subpermanents of $M$. 
\begin{defn}\label{genericM}
Let $\k$ be a field of characteristic zero, and let $M$ be an $m \times n$ matrix with entries $M_{ij} = x_{ij} \in R = \k[x_{11} \vvirg x_{mn}]$. The $k$-th generic determinantal ideal $D_k(M)$ is the ideal generated by the $k \times k$ subdeterminants of $M$. The $k$-th generic permanental ideal $P_k(M)$ is the ideal generated by the $k \times k$ subpermanents of $M$.
\end{defn} 
The zero set of the elements of the $k$-th generic determinantal ideal is the locus, in the space of all $m \times n$ matrices, of matrices of rank less than $k$. These ideals have been intensively studied; Eagon-Northcott determined the free resolution for maximal minors in \cite{EN}, and Lascoux solved the question for arbitrary $k$ in \cite{L}. In contrast, generic permanental ideals have received relatively little attention. In \cite{LS}, Laubenbacher-Swanson describe a Gr\"obner basis for the ideal of $2 \times 2$ permanents, and in \cite{K}, Kirkup determines the minimal primes of the ideal of $3 \times 3$ permanents. 

Permanental ideals arise in several contexts: for a bipartite graph $G$, the permanent of the adjacency matrix of $G$ counts the number of perfect matchings of $G$; in the study of matrices over a finite field they appear in the Alon-Tarsi conjecture \cite{AT}. The recent interest stems from computational complexity: the best known algorithm to compute the permanent is exponential, and the permanent plays a central role in Valiant's conjecture \cite{V}, which in turn spurred the development of geometric complexity theory by Mulmuley and Sohoni \cite{Mul1,Mul2, Mul3}. We discuss some of these motivations in \autoref{sec: motivations}. In this work, we investigate the minimal free resolution of the simplest nontrivial permanental ideal, in the case $k = m = 2$, that is, the minimal free resolution of the $2 \times 2$ permanents of a $2 \times n$ matrix of variables. For a $2 \times n$ matrix $M$ as in Definition~\ref{genericM}, we use $P_{2 \times n}$ to denote $P_2(M)$.

\begin{thm}\label{main}
For every $n \geq 3$, let  
\[
F_\bullet : 0 \to \cdots \to F_i \to F_{i-1} \to  \cdots \to F_1 \to R \to R / P_{2 \times n} \to 0
\]
be the minimal free resolution of $R/P_{2 \times n}$, as an $R$-module. Let $b_{ij}$ be the graded Betti numbers of $F_\bullet$, that is $F_i = \bigoplus_{j \geq 0} R(-j)^{b_{i, j}}$. Then
\begin{align*}
b_{1,2} = \binom{n}{2}, \qquad b_{i,i+1} =& 0 \text{ if $i > 1$}, \\
b_{i,i+2} =& \left[2 \binom{n}{i+2} - \binom{2n}{i+2} + \binom{n+1}{2} \binom{2n-2}{i} - 2 \binom{n}{2}\binom{2n-3}{i-1} \right]  \\ +&  \sum_{w = 3}^{\lfloor\frac{i+2}{2}\rfloor} 2^{i+2-2w} \binom{n}{w} \binom{n-w}{i+2-2w} \binom{w-1}{2} , \\
b_{i,i+3} =& \sum_{w = 3}^{\lfloor\frac{i+3}{2}\rfloor} 2^{i+3-2w} \binom{n}{w} \binom{n-w}{i+3-2w} \binom{w-1}{2},\\
b_{i,j} =& 0 \text{ otherwise},
\end{align*}
with the convention that $\binom{a}{b} = 0$ if $b < 0$ or $b > a$.
In particular, $R/P_{2 \times n}$ has regularity three and projective dimension $2n-3$. 
\end{thm}

\subsection{Motivations}\label{sec: motivations}
We briefly outline the role of algebraic and geometric methods in the study of Valiant's complexity classes in algebraic complexity theory. We refer to \cite{BLMW} and \cite[Ch.6]{Lan:GeometryComplThBook} for details.

Valiant's flagship conjecture in algebraic complexity theory states the the complexity classes $\bfVBP$ and $\bfVNP$ are distinct \cite{V}: more precisely $\bfVNP \not \subseteq \bfVBP$. It has a purely algebraic formulation, in terms of the notion of determinantal complexity of the permanent polynomial $\perm_m$. The \emph{determinantal complexity} of a homogeneous polynomial $F \in \bbC[x_1 \vvirg x_M]$ of degree $d$ is the smallest integer $n$ such that $z^{n-d} F$ can be expressed as the determinant of an $n \times n$ matrix whose entries are linear forms in $x_1 \vvirg x_M, z$. Setting $V = \bbC[x_1 \vvirg x_M,z]_1$, this is equivalent to the fact that $z^{n-d} F \in \End(V) \cdot \det_n$, where $\det_n$ is the determinant polynomial in $n^2$ variables, and $\cdot$ denotes the natural action of $\End(V)$ on  $\bbC[x_1 \vvirg x_M,z]_n$. 

Valiant's conjecture is equivalent to the fact that the determinantal complexity of the permanent polynomial $\perm_m$ grows superpolynomially, as a function of $m$. In their GCT program \cite{Mul2}, Mulmuley and Sohoni proposed a representation-theoretic approach to a potentially stronger version of this conjecture: $\bfVNP \not \subseteq \bar{\bfVBP}$. This can be phrased similarly, using the notion of border determinantal complexity, which is the smallest $n$ such that $z^{n-d} F \in \bar{ \End(V) \cdot \det_n}$, where the overline indicates closure in the Zariski topology.

This conjecture motivated the study of geometric properties of the permanent and determinant polynomials, in order to determine obstructions to the containment $z^{n-m} \perm_m \in \bar{ \End(V) \cdot \det_n}$. Strong barriers have been shown for the representation-theoretic method proposed in the GCT program \cite{Mul2,IkPa:Rectangular_Kron_in_GCT,BuIkPa:no_occurrence_obstructions_in_GCT}, and for other approaches, such as the method of shifted partials \cite{GKKS:ArithmeticCircuitsChasmDepthThree,elsw}. Importantly, these barriers rely on the presence of the \emph{padding} factor $z^{n-m}$ in the statement of the problem. There are however \emph{padding-free} formulations of the conjecture based on the \emph{iterated matrix multiplication polynomial} \cite{IkLan:Compl_of_perm_in_various_comp_models}, and some of the barriers proved in the padded setting extend to the padding-free setting \cite{GesIkPa:GCTMatrixPowering,GesLan:ExplicitPolysMaxPartialsMulmuley}.

This requires one to investigate alternative approaches to the conjecture, based on comparing more advanced invariants that can separate the (padded) permanent from the determinant, or from other polynomials which characterize the class $\bfVBP$, such as the iterated matrix multiplication polynomial. For instance, one can investigate the Betti numbers of the minimal free resolution of the Jacobian ideal. A similar approach has been employed in \cite{BoTe:WaringRankSyz}, in a slightly different setting, and it successfully provided lower bounds for the Waring rank of the determinant and the permanent polynomials, in small cases.

 A fundamental component of this approach consists in understanding the minimal free resolution of permanental ideals, which are the Jacobian ideals of the permanent polynomial. We study this problem in the first non-trivial case, which is the ideal of $2 \times 2$ subpermanents of a $2 \times n$ matrix; this is a first step in a program to obtain improved lower bounds in a range relevant for complexity theory.

\subsection{Resolutions and Betti tables}
Free resolutions of modules are a central object in commutative algebra. In this section, we recall the definition and some basic properties. We refer to \cite[Ch.1]{Eisenbud:SyzygyBook} for an introduction to the subject.

Let $S = \bbK[x_1 \vvirg x_N]$ denote the polynomial ring in $N$ variables endowed with its standard grading defined by $\deg(x_i) = 1$;  $S(d)$ denotes the ring $S$, regarded as a graded module over itself, with the grading shifted by $d$, that is $\deg(x_i) = 1+d$ as an element of $S(d)$. 

Let $M$ be a finitely generated graded module over $S$. A \emph{free resolution} of $M$ is an exact sequence 
\[
F_\bullet : 0 \to F_m \xto{\delta_{m}} \cdots F_2 \xto{\delta_2} F_1 \xto{\delta_1} F_0 \to M \to 0
\]
where $$F_i = \bigoplus S(-j)^{\oplus b_{ij}}$$ are free $S$-modules of finite rank. The maps $\delta_i : F_i \to F_{i-1}$ are the \emph{differentials} of the resolution: after fixing homogeneous bases of $F_i$ and $F_{i-1}$, the map $\delta_i$ can be regarded as a matrix whose entries are homogeneous polynomials.

The {\em Hilbert syzygy theorem} \cite{Eisenbud:SyzygyBook} guarantees that every finitely generated graded $S$-module has a finite free resolution. The integers $b_{ij}$ are called the \emph{graded Betti numbers} of the resolution. A free resolution is \emph{minimal} if the Betti numbers $b_{ij}$ are minimal among all free resolutions of $M$. A module $M$ admits a unique minimal free resolution, up to a certain notion of equivalence. In the following, we always assume that free resolutions are minimal. 

The Betti numbers of a minimal free resolution are uniquely determined by $M$ and are called the Betti numbers of $M$. The length $m$ of a minimal free resolution is called the \emph{projective dimension} of $M$ and denoted $\pdim(M)$. A consequence of the minimality property is that $b_{i,j} = 0$ if $j < i$. The value $\reg(M) = \max \{ j : b_{i,i+j} \neq 0\}$ is called the regularity of $M$. The Betti numbers are often recorded in a \emph{Betti table}, with the convention that the entry $(i,-j)$ of the table records the number $b_{i,i+j}$; the top row records the total rank of the corresponding free module. In this way $\pdim(M)$ and $\reg(M)$ are the number of columns and the number of rows of the Betti table, respectively. We illustrate this in an example. 

\begin{example}\label{example: small perm and det}
 Let $M = R/ P_{2 \times 4}$ be the coordinate ring of the ideal of $2 \times 2$ permanents of a $2 \times 4$ matrix. Its minimal free resolution turns out to be 
 \[
 0 \to \begin{array}{c}
 R(-8)^{\oplus 3} \\
 \end{array}\to \begin{array}{c}
 R(-6)^{\oplus 6} \\
 R(-7)^{\oplus 8} \\
 \end{array} \to
 \begin{array}{c}
 R(-5)^{\oplus 24} \\
 R(-6)^{\oplus 4} \\
 \end{array} 
 \to \begin{array}{c}
 R(-4)^{\oplus 22} \\
 \end{array} 
 \to  R(-2) ^{\oplus 6} \to  R \to M \to 0.
 \]
 For instance, the differential $\delta_3: F_4 \to F_3$ is represented by a $(24 + 4) \times (6 + 8)$ whose entries are homogeneous polynomial of degree $1$ and $2$; the matrix has a block structure so that it is a degree $0$ map of $R$-modules. The corresponding Betti table is 
 \[
\begin{matrix}
        & 0 & 1 & 2 & 3 & 4 & 5\\
     \text{total:}
        & 1 & 6 & 22 & 28 & 14 & 3\\
     0: & 1 & - & - & - & - & -\\
     1: & - & 6 & - & - & - & -\\
     2: & - & - & 22 & 24 & 6 & -\\
     3: & - & - & - & 4 & 8 & 3
     \end{matrix}
     \]
Therefore the regularity of $M$ is $3$ and its projective dimension is $5$.

For comparison, let $N = R/ D_2(M)$ be the determinantal ideal of $2 \times 2$ minors of $M$. Its free resolution is given by the Eagon-Northcott complex \cite[Sec. A2H]{Eisenbud:SyzygyBook}
\[
0 \to R(-4)^{\oplus 3} \to R(-3)^{\oplus 8} \to R(-2)^{\oplus 6} \to R \to N \to 0.
\]
All differential are matrices of linear forms. The corresponding Betti table is 
\[
\begin{matrix}
         & 0 & 1 & 2 & 3\\
      \text{total:}
         & 1 & 6 & 8 & 3\\
      0: & 1 & - & - & -\\
      1: & - & 6 & 8 & 3
      \end{matrix}
\]
Therefore $N$ has regularity $1$ and projective dimension $3$.
 \end{example}

\subsection{Permanents and Determinants}
As suggested in \autoref{example: small perm and det}, the situation for determinantal ideals is significantly easier. In fact, the minimal free resolution of the general determinantal ideal $D_k(M)$ was determined in full generality by Lascoux \cite{L}. When $k = m$, it is given by the {\em Eagon-Northcott} complex \cite[Sec. A2H]{Eisenbud:SyzygyBook}. The ideal $D_m(M)$ is prime and Cohen-Macaulay. The Cohen-Macaulay property implies that quotienting by a regular sequence preserves the free resolution \cite{E}. Many varieties commonly encountered in algebraic geometry are determinantal. 
\begin{example}\label{TC}
The $d^{th}$ Veronese embedding 
\[
\P^1 \rightarrow \P^d \mbox{ defined via }(s:t) \mapsto (s^d:s^{d-1}t:\cdots: st^{d-1}:t^d)
\]
has defining equations given by the vanishing of the $2 \times 2$ minors of 
\[
\left[ \!
\begin{array}{cccc}
x_0 & x_1 & \cdots &x_{d-1}\\
x_1 & x_2 & \cdots & x_d
\end{array}\! \right]
\]
While this matrix is not generic in the sense of \autoref{genericM}, replacing $x_i$ with $x_{d-1+i}$ in the bottom
row yields a generic matrix; an easy check shows the sequence $x_i-x_{d-1+i}$ is regular, and so the minimal free 
resolution will be an Eagon-Northcott complex. 
\end{example} 
Ideals generated by the minors of a suitably generic matrix arise in many contexts: 
\begin{itemize}
    \item In \cite{e1}, Eisenbud relates decomposition a divisor $D = E+F$ on a curve $C$ to determinantal equations for $I_C$. 
    \item In \cite{SidS}, Sidman-Smith show that for any projective variety $X$, there is a line bundle $D$ such that for all $E$ with $E\otimes D^\vee$ ample, the ideal of the image of $X$ is determinantally presented, extending a result of Mumford. 
    \item In \cite{BGL}, Buczy\'{n}ski-Ginensky-Landsberg relate equations of secant varieties to determinantal equations. 
        \end{itemize}

\subsection{Outline of the proof and structure of the paper}

To prove \autoref{main}, we exploit the close connection between homological properties of an ideal and the ones of its initial ideal. We observe in \autoref{sec: simplicial complex inP2} that the initial ideal $\init(P_{2 \times n})$ of $P_{2 \times n}$, with respect to a certain monomial order, is a squarefree monomial ideal. In \autoref{CMreg}, we use Stanley-Reisner Theory and Alexander duality to prove that the regularity of $R/ \init(P_{2 \times n})$ is $3$; in particular, the only nonzero rows of the Betti table of $P_{2 \times n}$ are the first, second and third.
\begin{itemize}
 \item For the first row of the Betti table, it suffices to show that $P_{2 \times n}$ has no linear syzygies. This follows from \cite{elsw}; we include a proof in \autoref{rmk: first row}.
\item In \autoref{lemma:thirdRowInitial} we use \emph{Hochster's formula} (\autoref{Hochster}) to compute the third row of the Betti table of the initial ideal $\init(P_{2 \times n})$.
\item For the second row of the Betti table, we first use a Hilbert series argument, together with standard semicontinuity properties of the Betti numbers, to obtain a formula for the differences between the Betti numbers of $P_{2 \times n}$ in its second and third row: this is given in \autoref{hilbSeries}.
\item To complete the proof of \autoref{main}, we show that the third rows of the Betti tables of $P_{2 \times n}$ and $\init(P_{2 \times n})$ coincide. This is obtained in \autoref{sec: spectral sequence} via a spectral sequence argument. We use the Bernstein-Gelfand-Gelfand correspondence to obtain an equivalent statement in terms of the homology of certain complexes of modules over an exterior algebra. We use such complexes to build a double complex in \autoref{lemma: commuting diagram}, and in \autoref{thm: spectral sequence count} we analyze the resulting spectral sequence to conclude. 
\end{itemize}
\section{A simplicial complex encoding the initial ideal of size two permanents}\label{sec: simplicial complex inP2}

A number of properties of an ideal can be investigated by studying its associated initial ideal, with respect to some monomial order. We refer to \cite[Ch.15]{Ecommalg} for the theory and we will mention throughout this section the important theoretical results that we will use.

In \cite[Theorem 3.1]{LS}, Laubenbacher-Swanson determine a Gr\"obner basis for the ideal of $2 \times 2$ permanents with respect to a diagonal monomial order. We state their result here in the particular case of a generic $2 \times n$ matrix, and in terms of an \emph{antidiagonal} order, that is a monomial order such that the leading term of any subpermanent is the product of the antidiagonal entries. 
\begin{thm}[{\cite[Theorem 3.1]{LS}}]\label{thm: gb permanents}
Let $P_{2 \times n}\subseteq \bbK[x_{11} \vvirg x_{2n}]$ be the ideal of $2\times 2$ permanents of the matrix 
\[
\left[\begin{array}{cccc}
x_{11} & \cdots &x_{1n}\\
x_{21} & \cdots &x_{2n}
\end{array} \right].
\]
A reduced Gr\"obner basis of $P_{2 \times n}$ with respect to any antidiagonal monomial order consists of the following three families of polynomials:
\begin{align*}
&x_{1i}x_{2j}+x_{1j}x_{2i} \quad \text{for $i > j$}, \\
&x_{1i}x_{2j}x_{2k} \quad \text{for $i< j < k$}, \\
&x_{1i}x_{1j}x_{2k} \quad \text{for $i< j < k$}.
\end{align*}
\end{thm}
An immediate consequence of \autoref{thm: gb permanents} is that the initial ideal of $P_{2 \times n}$ with respect to any antidiagonal monomial order is 
\[
\init(P_{2 \times n})= \begin{cases} x_{1i}x_{2j} & i > j, \\
x_{1i}x_{2j}x_{2k} & i< j < k, \\
x_{1i}x_{1j}x_{2k} & i< j < k.
\end{cases}
\]
In particular, $\init(P_{2 \times n})$ is a squarefree monomial ideal. We will determine the regularity of $P_{2 \times n}$ by studying $\init(P_{2 \times n})$ via {\em Stanley-Reisner theory}.

\subsection{Stanley-Reisner ring and Alexander duality}\label{sec: SR and Alexander}
Many homological properties of a squarefree monomial ideal are determined by the combinatorial structure of an associated simplicial complex. 
\begin{defn}\label{SRing}
Let $\Delta$ be a simplicial complex on the vertex set $V=\{1 \vvirg \nu\}$ with $|V|=\nu$. The \emph{Stanley-Reisner ideal} ideal of $\Delta$ is 
\[
I_\Delta = \langle x_{i_1}\cdots x_{i_m} \mid [i_1,\ldots,i_m] \mbox{ is not a face of }\Delta \rangle \subseteq \bbK[x_1 \vvirg x_\nu]
\]
and the \emph{Stanley-Reisner ring} of $\Delta$ is $\mathbb{K}[x_1,\ldots,x_\nu]/I_\Delta$.
\end{defn}
It is easy to see that the Stanley-Reisner ideal defines a one-to-one correspondence between squarefree monomial ideals and simplicial complexes. Stanley-Reisner theory establishes a correspondence between the algebraic properties of the ideal $I_\Delta$ and the combinatorial properties of $\Delta$, regarded as topological space. We refer to \cite[Part I]{MS} and \cite[Ch.5]{S} for an introduction on this subject. We record an easy, important fact. A \emph{coface} of $\Delta$ is the complement of a face of $\Delta$; let $\CF(\Delta)$ be the set of minimal cofaces of $\Delta$. 
\begin{thm}[{\cite[Thm. 5.3.3]{S}}]\label{primaryDecompSR}
Let $\Delta$ be a simplicial complex. The primary decomposition of the ideal $I_\Delta$ is
\[
I_\Delta = \bigcap\limits_{[v_{i_1} \cdots v_{i_k}] \in \CF(\Delta)}\langle x_{i_1}, \ldots, x_{i_k} \rangle.
\]
\end{thm}
\begin{example}\label{SRexm}
Let $\Delta$ be a simplicial complex on the set $V = \{ 1 \vvirg 4\}$ of four vertices, with maximal faces given by the edge $[12]$ and the triangle $[234]$. The non-faces of $\Delta$ are the edges $[13],[14]$, the triangles $[123], [124], [134]$, and the simplex $[1234]$. In particular, the minimal non-faces are two edges $[13],[14]$, and the Stanley-Reisner ideal is 
\[
  I_\Delta = \langle x_1x_3, x_1x_4 \rangle.
\]
The minimal cofaces of $\Delta$ are the complement of its maximal faces, that is $\CF(\Delta) = \{ \bar{[234]}, \bar{[12]} \} = \{ [1],[34]\}$. The primary decomposition of $I_\Delta$ is 
  \begin{equation}\label{PDSR}
    I_\Delta = \langle x_1\rangle \cap \langle x_3,x_4\rangle. 
  \end{equation}
  as predicted by \autoref{primaryDecompSR}.
\end{example}
By \autoref{primaryDecompSR}, the primary components of a square-free monomial ideal are minimally generated by a subset of the variables. Therefore, the product of the minimal generators is a square-free monomial. The ideal generated by such monomials (one for each primary component) is called the \emph{monomialization} of the primary decomposition of $I_\Delta$. This resulting monomial ideal is closely related to $I_\Delta$. To explain this relation, we introduce the \emph{Alexander dual} of a simplicial complex.
\begin{defn}\label{Adual}
Let $\Delta$ be a simplicial complex. The Alexander dual $\Delta^\vee$ of $\Delta$ is the simplicial complex
\[
\Delta^\vee = \{\tau \mid \overline{\tau} \not\in \Delta \}
\]
where $\overline{\tau}$ is the complement of $\tau$ in $V$, that is $\bar{\tau} =V \setminus \tau$. 
\end{defn}

The next result is fundamental. We refer to \cite{ER} or \cite{H} for a proof.
\begin{thm}\label{AlexPrimary}
The monomialization of the primary decomposition of $I_\Delta$ is the Stanley-Reisner ideal $I_{\Delta^\vee}$. 
\end{thm}
\begin{example}\label{SRexm3} Recall the simplicial complex $\Delta$ of \autoref{SRexm}, which is the complex with maximal faces $\{ [12],[234]\}$. The nonfaces of $\Delta$ are
\[
\{[13], [14],[123],[124], [134], [1234]\},
\]
whose complements
\[
\{ [24], [23],[4],[3],[2], \emptyset\}, 
\]
define the faces of $\Delta^\vee$. Hence, the maximal faces of $\Delta^\vee$ are the edges $[23]$ and $[24]$. Therefore, the nonfaces of $\Delta^\vee$ are $\{ [1], [34]\}$ yielding the associated Stanley-Reisner ideal $I_{\Delta^\vee} = \langle x_1,x_3x_4\rangle$. Observe, that this is exactly the monomialization of the primary decomposition of $I_\Delta$ in \autoref{PDSR}, as predicted by \autoref{AlexPrimary}. In fact, the generators of $I_{\Delta^\vee}$ can be obtained directly from the maximal faces of $\Delta$ as follows: let $P = \prod_1^N x_i$ be the monomial obtained by multiplying all the variables in the polynomial ring of interest; then
\[
I_{\Delta^\vee} = \left( P /( x_{j_1} \cdots x_{j_k}) : [j_1 \vvirg j_k]  \in \Delta \right).
\]
\end{example}

Our goal is to realize the Betti numbers as ranks of certain restrictions of the homology groups of  the homology group of the Stanley-Reisner complex $\Delta_{P_{2 \times n}}$ associated to the ideal $\init(P_{2 \times n})$. To do this, we consider the \emph{fine grading} of the polynomial ring, that is the $\Z^\nu$-grading on $R=\k[x_1,\ldots,x_\nu]$ defined by $\deg(x_i)=\bfe_i$. In this setting \emph{Hochster's formula} (see \autoref{Hochster}) relates the multigraded Betti numbers of a squarefree monomial ideal to the homology of the corresponding simplicial complex.

In order to keep track of the grading, we use the following characterization of the Betti numbers. Regard $\bbK$ as a trivial $S$-module: the copy of $\bbK$ in the degree $0$ component of $S$ acts by multiplication, whereas all positive degree components act by multiplication by $0$. Tensoring the minimal free resolution of a module $M$ by $\bbK$ provides a complex $F_\bullet \otimes_S \bbK$ of $S$-modules; the minimality of the resolution guarantees that all differentials are identically $0$. The homology groups of this complex are graded vector spaces, and the graded components are denoted by 
\[
\Tor_i(M, \bbK)_{j} = (F_i \otimes \bbK)_j 
\]
where the subscript $j$ denotes the homogeneous component of degree $j$. In particular, $b_{ij} = \dim \Tor_i(M, \bbK)_{j}$. In fact, if the polynomial ring is endowed with a different grading, the $\Tor$-modules are graded with respect to such grading as well. With this notation, we can state \emph{Hochster's formula}:
\begin{thm}[{\cite[Cor.~5.12]{MS}}, \cite{Reisner}]\label{Hochster}
Let $\Delta$ be a simplicial complex on $\nu$ vertices. Then the nonzero Betti numbers of $R/I_{\Delta}$ satisfy
\[
\Tor_i (R/I_{\Delta}, \bbK)_\sigma \simeq \tilde{\rmH}^{|\sigma|-i-1}(\Delta |_{\sigma},\mathbb{K}),
\]
and lie only in squarefree degrees $\sigma \in \Z^\nu$. In particular $\Tor_i (R/I_{\Delta}, \bbK)_\sigma = 0$ if $\sigma \notin \{ 0,1\}^\nu$.
 \end{thm}
We also record a result that relates the regularity of a Stanley-Reisner ideal with the projective dimension of its Alexander dual.
\begin{thm}[{\cite[Cor.~2.9]{BCP}}]\label{regpd}
The regularity of $I_\Delta$ equals the projective dimension of $R/I_{\Delta^\vee}$.
\end{thm}
Using these facts, we are able to prove the main theorem of this section.
\begin{thm}\label{CMreg}
The regularity of $R/\init(P_{2 \times n})$ is at most three. 
\end{thm}
\begin{proof}
Since $\init(P_{2 \times n})$ is a squarefree monomial ideal, it is the Stanley-Reisner ideal of a simplicial complex $\Delta$. Since $\reg(R/I) = \reg(I)-1$, by \autoref{regpd} it suffices to show 
that $\pdim(R/I_{\Delta^\vee})\le 4$. 

From the description of the monomials defining $\init(P_{2 \times n})$, we can deduce that $\Delta$ is the simplicial complex on the vertex set $\{ x_{11} \vvirg x_{2n}\}$ whose maximal faces are described as follows:
\begin{itemize}[leftmargin=*]
    \item two $(n-1)$-simplices $\Delta_1$ and $\Delta_2$ on the sets $\{x_{11},\ldots, x_{1n}\}$ and $\{x_{21},\ldots, x_{2n}\}$, respectively;
    \item triangles of the form $\{x_{1i},x_{2i},x_{2j}\}$ and $\{x_{1i},x_{1j},x_{2j}\}$, with $i < j$.
\end{itemize} 
A representation of $\Delta$ in the case $n=3$ is given in \autoref{fig:Delta3}. In particular the maximal faces of $\Delta$
are $\Delta_1$, $\Delta_2$, and $n(n-1)$ triangles. Let $P = x_{11} \cdots x_{1n}x_{21} \cdots x_{2n}$ be the product of variables. The ideal $I_{\Delta^\vee}$ is generated by $x_{11} \cdots x_{1n}$, $x_{21} \cdots x_{2n}$ and $n(n-1)$ monomials
$P/x_{1i}x_{2i}x_{2j}$, $P/x_{1i}x_{1j}x_{2j}$.

By \autoref{Hochster}, the nonzero Betti numbers of $I_{\Delta^\vee}$ only occur in squarefree multidegrees, so the only syzygies involving the generators $P/x_{1i}x_{2i}x_{2j}$, $P/x_{1i}x_{1j}x_{2j}$ stop in projective dimension at most $3$; the only other generators are 
$x_{11} \cdots x_{1n}$, $x_{21} \cdots x_{2n}$, which are a complete intersection, and their only syzygy is the Koszul syzygy in projective dimension two.\end{proof}

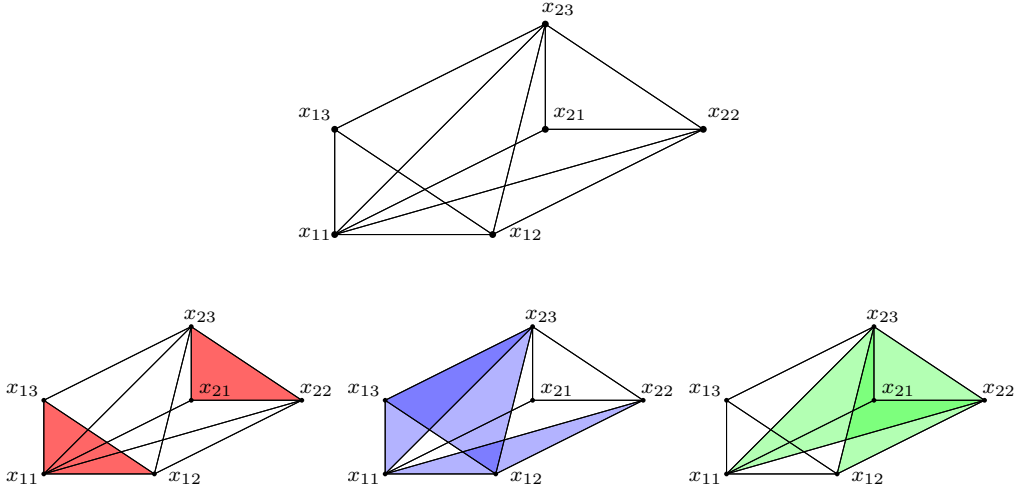
\begin{figure}
    \centering
    \begin{tikzpicture}[line cap=round,line join=round,>=triangle 45,x=.7cm,y=.7cm]
\clip(-2,-2) rectangle (7,4);
\draw [color=black] (-1,-1)-- (2,-1);
\draw [color=black] (2,-1)-- (-1,1);
\draw [color=black] (-1,1)-- (-1,-1);
\draw [color=black] (3,1)-- (6,1);
\draw [color=black] (6,1)-- (3,3);
\draw [color=black] (3,3)-- (3,1);
\draw [color=black] (-1,-1)-- (2,-1);
\draw [color=black] (2,-1)-- (6,1);
\draw [color=black] (6,1)-- (-1,-1);
\draw [color=black] (-1,-1)-- (-1,1);
\draw [color=black] (-1,1)-- (3,3);
\draw [color=black] (3,3)-- (-1,-1);
\draw [color=black] (2,-1)-- (-1,1);
\draw [color=black] (-1,1)-- (3,3);
\draw [color=black] (3,3)-- (2,-1);
\draw [color=black] (-1,-1)-- (3,1);
\draw [color=black] (3,1)-- (6,1);
\draw [color=black] (6,1)-- (-1,-1);
\draw [color=black] (2,-1)-- (6,1);
\draw [color=black] (6,1)-- (3,3);
\draw [color=black] (3,3)-- (2,-1);
\draw [color=black] (-1,-1)-- (3,1);
\draw [color=black] (3,1)-- (3,3);
\draw [color=black] (3,3)-- (-1,-1);
\begin{scriptsize}
\fill [color=black] (3,1) circle (1.3pt);
\draw[color=black,anchor= west] (3,1.29) node {$x_{21}$};
\fill [color=black] (-1,-1) circle (1.3pt);
\draw[color=black,anchor=north east] (-0.9,-0.73) node {$x_{11}$};
\fill [color=black] (2,-1) circle (1.3pt);
\draw[color=black,anchor=north west] (2.17,-0.73) node {$x_{12}$};
\fill [color=black] (-1,1) circle (1.3pt);
\draw[color=black,anchor=east] (-0.9,1.29) node {$x_{13}$};
\fill [color=black] (3,3) circle (1.3pt);
\draw[color=black] (3.25,3.28) node {$x_{23}$};
\fill [color=black] (6,1) circle (1.3pt);
\draw[color=black] (6.4,1.29) node {$x_{22}$};
\end{scriptsize}
\end{tikzpicture}\\
\begin{tikzpicture}[line cap=round,line join=round,>=triangle 45,x=.7cm,y=.7cm,scale=.7]
\clip(-2,-2) rectangle (7,4);
\fill[color=red,fill=red,fill opacity=0.6] (-1,-1) -- (2,-1) -- (-1,1) -- cycle;
\fill[color=red,fill=red,fill opacity=0.6] (3,1) -- (6,1) -- (3,3) -- cycle;
\draw [color=black] (-1,-1)-- (2,-1);
\draw [color=black] (2,-1)-- (-1,1);
\draw [color=black] (-1,1)-- (-1,-1);
\draw [color=black] (3,1)-- (6,1);
\draw [color=black] (6,1)-- (3,3);
\draw [color=black] (3,3)-- (3,1);
\draw [color=black] (-1,-1)-- (2,-1);
\draw [color=black] (2,-1)-- (6,1);
\draw [color=black] (6,1)-- (-1,-1);
\draw [color=black] (-1,-1)-- (-1,1);
\draw [color=black] (-1,1)-- (3,3);
\draw [color=black] (3,3)-- (-1,-1);
\draw [color=black] (2,-1)-- (-1,1);
\draw [color=black] (-1,1)-- (3,3);
\draw [color=black] (3,3)-- (2,-1);
\draw [color=black] (-1,-1)-- (3,1);
\draw [color=black] (3,1)-- (6,1);
\draw [color=black] (6,1)-- (-1,-1);
\draw [color=black] (2,-1)-- (6,1);
\draw [color=black] (6,1)-- (3,3);
\draw [color=black] (3,3)-- (2,-1);
\draw [color=black] (-1,-1)-- (3,1);
\draw [color=black] (3,1)-- (3,3);
\draw [color=black] (3,3)-- (-1,-1);
\begin{scriptsize}
\fill [color=black] (3,1) circle (1.3pt);
\draw[color=black,anchor= west] (3,1.29) node {$x_{21}$};
\fill [color=black] (-1,-1) circle (1.3pt);
\draw[color=black,anchor=north east] (-0.9,-0.73) node {$x_{11}$};
\fill [color=black] (2,-1) circle (1.3pt);
\draw[color=black,anchor=north west] (2.17,-0.73) node {$x_{12}$};
\fill [color=black] (-1,1) circle (1.3pt);
\draw[color=black,anchor=east] (-0.9,1.29) node {$x_{13}$};
\fill [color=black] (3,3) circle (1.3pt);
\draw[color=black] (3.25,3.28) node {$x_{23}$};
\fill [color=black] (6,1) circle (1.3pt);
\draw[color=black] (6.4,1.29) node {$x_{22}$};
\end{scriptsize}
\end{tikzpicture}
\begin{tikzpicture}[line cap=round,line join=round,>=triangle 45,x=.7cm,y=.7cm,scale=.7]
\clip(-2,-2) rectangle (7,4);
\fill[color=blue,fill=blue,fill opacity=0.3] (-1,-1) -- (2,-1) -- (6,1) -- cycle;
\fill[color=blue,fill=blue,fill opacity=0.3] (-1,-1) -- (-1,1) -- (3,3) -- cycle;
\fill[color=blue,fill=blue,fill opacity=0.3] (2,-1) -- (-1,1) -- (3,3) -- cycle;
\draw [color=black] (-1,-1)-- (2,-1);
\draw [color=black] (2,-1)-- (-1,1);
\draw [color=black] (-1,1)-- (-1,-1);
\draw [color=black] (3,1)-- (6,1);
\draw [color=black] (6,1)-- (3,3);
\draw [color=black] (3,3)-- (3,1);
\draw [color=black] (-1,-1)-- (2,-1);
\draw [color=black] (2,-1)-- (6,1);
\draw [color=black] (6,1)-- (-1,-1);
\draw [color=black] (-1,-1)-- (-1,1);
\draw [color=black] (-1,1)-- (3,3);
\draw [color=black] (3,3)-- (-1,-1);
\draw [color=black] (2,-1)-- (-1,1);
\draw [color=black] (-1,1)-- (3,3);
\draw [color=black] (3,3)-- (2,-1);
\draw [color=black] (-1,-1)-- (3,1);
\draw [color=black] (3,1)-- (6,1);
\draw [color=black] (6,1)-- (-1,-1);
\draw [color=black] (2,-1)-- (6,1);
\draw [color=black] (6,1)-- (3,3);
\draw [color=black] (3,3)-- (2,-1);
\draw [color=black] (-1,-1)-- (3,1);
\draw [color=black] (3,1)-- (3,3);
\draw [color=black] (3,3)-- (-1,-1);
\begin{scriptsize}
\fill [color=black] (3,1) circle (1.3pt);
\draw[color=black,anchor= west] (3,1.29) node {$x_{21}$};
\fill [color=black] (-1,-1) circle (1.3pt);
\draw[color=black,anchor=north east] (-0.9,-0.73) node {$x_{11}$};
\fill [color=black] (2,-1) circle (1.3pt);
\draw[color=black,anchor=north west] (2.17,-0.73) node {$x_{12}$};
\fill [color=black] (-1,1) circle (1.3pt);
\draw[color=black,anchor=east] (-0.9,1.29) node {$x_{13}$};
\fill [color=black] (3,3) circle (1.3pt);
\draw[color=black] (3.25,3.28) node {$x_{23}$};
\fill [color=black] (6,1) circle (1.3pt);
\draw[color=black] (6.4,1.29) node {$x_{22}$};
\end{scriptsize}
\end{tikzpicture}
\begin{tikzpicture}[line cap=round,line join=round,>=triangle 45,x=.7cm,y=.7cm,scale=.7]
\clip(-2,-2) rectangle (7,4);
\fill[color=green,fill=green,fill opacity=0.3] (-1,-1) -- (3,1) -- (6,1) -- cycle;
\fill[color=green,fill=green,fill opacity=0.3] (2,-1) -- (6,1) -- (3,3) -- cycle;
\fill[color=green,fill=green,fill opacity=0.3] (-1,-1) -- (3,1) -- (3,3) -- cycle;
\draw [color=black] (-1,-1)-- (2,-1);
\draw [color=black] (2,-1)-- (-1,1);
\draw [color=black] (-1,1)-- (-1,-1);
\draw [color=black] (3,1)-- (6,1);
\draw [color=black] (6,1)-- (3,3);
\draw [color=black] (3,3)-- (3,1);
\draw [color=black] (-1,-1)-- (2,-1);
\draw [color=black] (2,-1)-- (6,1);
\draw [color=black] (6,1)-- (-1,-1);
\draw [color=black] (-1,-1)-- (-1,1);
\draw [color=black] (-1,1)-- (3,3);
\draw [color=black] (3,3)-- (-1,-1);
\draw [color=black] (2,-1)-- (-1,1);
\draw [color=black] (-1,1)-- (3,3);
\draw [color=black] (3,3)-- (2,-1);
\draw [color=black] (-1,-1)-- (3,1);
\draw [color=black] (3,1)-- (6,1);
\draw [color=black] (6,1)-- (-1,-1);
\draw [color=black] (2,-1)-- (6,1);
\draw [color=black] (6,1)-- (3,3);
\draw [color=black] (3,3)-- (2,-1);
\draw [color=black] (-1,-1)-- (3,1);
\draw [color=black] (3,1)-- (3,3);
\draw [color=black] (3,3)-- (-1,-1);
\begin{scriptsize}
\fill [color=black] (3,1) circle (1.3pt);
\draw[color=black,anchor= west] (3,1.29) node {$x_{21}$};
\fill [color=black] (-1,-1) circle (1.3pt);
\draw[color=black,anchor=north east] (-0.9,-0.73) node {$x_{11}$};
\fill [color=black] (2,-1) circle (1.3pt);
\draw[color=black,anchor=north west] (2.17,-0.73) node {$x_{12}$};
\fill [color=black] (-1,1) circle (1.3pt);
\draw[color=black,anchor=east] (-0.9,1.29) node {$x_{13}$};
\fill [color=black] (3,3) circle (1.3pt);
\draw[color=black] (3.25,3.28) node {$x_{23}$};
\fill [color=black] (6,1) circle (1.3pt);
\draw[color=black] (6.4,1.29) node {$x_{22}$};
\end{scriptsize}
\end{tikzpicture}
    \caption{The simplicial complex $\Delta$ for $n=3$. Top: The $1$-skeleton. Bottom: The three types of higher dimensional faces: in red the two simplices $\Delta_1,\Delta_2$, in blue the $2$-faces $x_{1i}x_{1j}x_{2j}$ with $i < j$, in green the $2$-faces $x_{1i}x_{2i}x_{2j}$ with $i < j$.}\label{fig:Delta3}
\end{figure}
The $f$-vector of a simplicial complex $\Delta$ of dimension $N$ is $f(\Delta) = [f_{-1} \vvirg f_N]$, where $f_i$ is the number of $i$-dimensional faces of $\Delta$. It is often convenient to record the $f$-vector in terms of its generating function $f(\Delta,t) = \sum_{k = 0}^{N+1} f_{k-1} t^{N+1-k}$.
\begin{lem}\label{lemma:f vector Delta}
Let $\Delta$ be the simplicial complex associated to $\init(P_{2 \times n})$. Then we have 
\begin{align*} 
f(\Delta)_{-1} &= 1\\
f(\Delta)_0 &= 2n, \\
f(\Delta)_1 &= 2 \textstyle\binom{n}{2} + \textstyle\binom{n+1}{2}, \\
f(\Delta)_2 &=  2\textstyle\binom{n}{3} + 2 \textstyle\binom{n}{2}, \\
f(\Delta)_k &= 2 \textstyle\binom{n}{k+1} \quad \text{for $k = 3 \vvirg n-1$}.
\end{align*}
In particular
\[
f(\Delta,t) = 2(t+1)^n - t^n + \textstyle \binom{n+1}{2} t^{n-2} + 2 \textstyle\binom{n}{2} t^{n-3}.
\]
\end{lem}
\begin{proof}
The values of $f(\Delta)_{-1}$ and $f(\Delta)_0$ are immediate, because the vertex set of $\Delta$ consists of $2n$ elements. 
    
There are $\binom{n}{2}$ edges contained in the simplex $\Delta_1$ and $\binom{n}{2}$ contained in the simplex $\Delta_2$. Moreover, for every choice of $i \leq j$, there is an edge $\{ x_{1i},x_{2j}\}$, resulting in $\binom{n+1}{2}$ additional edges. Therefore $f(\Delta)_1 = 2 \binom{n}{2} + \binom{n+1}{2}$.

There are $\binom{n}{3}$ triangles contained in the simplex $\Delta_1$ and $\binom{n}{3}$ contained in the simplex $\Delta_2$. Moreover, for every $i < j$ there are two triangles  $\{ x_{1i},x_{2i},x_{2,j}\}$  and $\{ x_{1i},x_{1j},x_{2,j}\}$, resulting in $2\binom{n}{2}$ additional triangles. Therefore $f(\Delta)_2 = 2\binom{n}{3} + 2 \binom{n}{2}$.

Finally, for $k =3 \vvirg n-1$, the only faces of dimension $k$ are those contained in either $\Delta_1$ or $\Delta_2$. Therefore $f(\Delta)_k = 2 \binom{n}{k+1}$ for $k > 2$.

Using $\dim \Delta = n-1$, we conclude that 
\begin{align*}
f(\Delta,t)= &\mbox{ }\textstyle\sum_{k = 0}^{(n-1)+1} f_{k-1}(\Delta)_j  \cdot t^{(n-1)+1-k} \\
=&\mbox{ }t^n + (2n) t^{n-1} + \left( 2 \textstyle\binom{n}{2} + \textstyle\binom{n+1}{2} \right) t^{n-2} + 
\left(2\textstyle\binom{n}{3} + 2 \textstyle\binom{n}{2} \right) t^{n-3} + \textstyle \sum_{k=4}^n \left(2 \textstyle\binom{n}{k+1}\right) t^{n-k} \\
=& - t^n + 2\textstyle\binom{n+1}{2} t^{n-2} + 2 \textstyle\binom{n}{2} t^{n-3} + \textstyle 2\sum_0^{n} \binom{n}{k} t^k   \\
=&- t^n + 2\textstyle\binom{n+1}{2} t^{n-2} + 2 \textstyle\binom{n}{2} t^{n-3} + 2(t+1)^n.
\end{align*}
This yields the desired expression for $f(\Delta,t)$.
\end{proof}

\subsection{The third row of the Betti table of the initial ideal}
By upper semicontinuity, the Betti numbers of $\init(P_{2 \times n})$ bound from above the Betti numbers for $P_{2 \times n}$. Moreover, by \autoref{CMreg}, we know there are only three strands in the Betti table for $R/I_\Delta$, where $I_\Delta = \init(P_{2 \times n})$; hence, the same holds for $P_{2 \times n}$. In this section, we use \autoref{Hochster} and a combinatorial argument to compute the third row of the Betti table of $R/\init(P_{2 \times n})$. The results of \autoref{sec: spectral sequence} will guarantee that these Betti numbers are the same as the ones of $R/P_{2 \times n}$. 

\begin{thm}\label{lemma:thirdRowInitial}
The Betti numbers in the third row of the Betti table of $\init(P_{2 \times n})$ are:
\[
b_{k,k+3} = \sum_{w = 3}^{\lfloor\frac{k+3}{2}\rfloor} 2^{k+3-2w} \binom{n}{w} \binom{n-w}{k+3-2w} \binom{w-1}{2}.
\]
Here we adopt the convention that $\binom{a}{b} = 0$ if $a < b$ or $b < 0$.
\end{thm}
\begin{proof}
Let $\Delta$ be the Stanley-Reisner complex of $\init(P_{2 \times n})$ described in \autoref{CMreg}. By \autoref{Hochster} we have
\[
 b_{k, k+3} = \dim \Tor_{k} (R/ I_\Delta)_{k+3} = \sum _{\sigma : |\sigma| = {k+3}} \dim \Tor_{k} (R/ I_\Delta)_{\sigma} = \sum _{\sigma : |\sigma| = k+3} \dim \rmH^2( \Delta|_\sigma),
\]
where $\sigma$ ranges subsets of variables $\{ x_{11} \vvirg x_{2n}\}$ of cardinality $k + 3$ or equivalently over the squarefree multidegrees of total degree $k + 3$. Therefore, we need to compute the dimensions $\rmH^2( \Delta|_\sigma)$ for induced subcomplexes $\Delta|_\sigma$. 

As observed in \autoref{CMreg}, $\Delta$ is the union of the simplices $\Delta_{1}$, $\Delta_{2}$, and the triangles 
\[
\{x_{1i},x_{2i},x_{2j}\} \mbox{ and } \{x_{1i},x_{1j},x_{2j}\} \mbox{ with } i < j, 
\]
in \autoref{fig:triangles with diagonal edge}.
\begin{figure}[h]
\vskip -.1in
\includegraphics[width=6in]{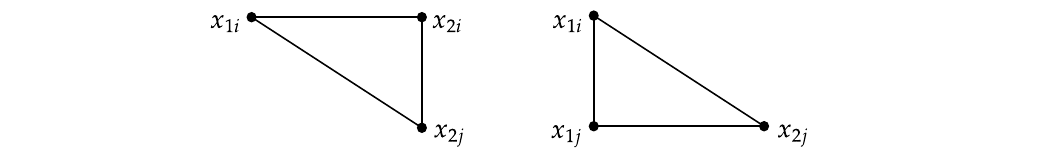} 
\vskip -.1in
\caption{A diagonal edge can only appear in two ways.}\label{fig:triangles with diagonal edge}
\end{figure}
Since $\Delta_1,\Delta_2$ are contractible, the cohomology is generated by the triangles in \autoref{fig:triangles with diagonal edge}. Every \emph{diagonal edge} $x_{1i},x_{2j}$ for $i < j$ appears in exactly two triangles, therefore if one of them contributes to a non-trivial cohomology cycle, then the other must appear as well. We deduce that the cohomology of $\rmH^2(\Delta)$ is generated by \emph{rectangles}, obtained by glueing the triangles of \autoref{fig:triangles with diagonal edge} along their diagonal edge, see \autoref{fig: rectangles}. 
\begin{figure}[h]
\vskip -.1in
\includegraphics[width=6in]{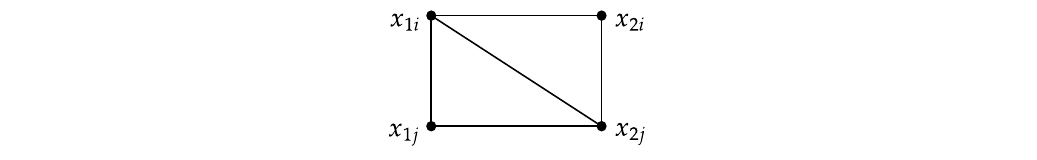} 
\vskip -.1in
\caption{A diagonal edge appears in at most two triangles}\label{fig: rectangles}
\end{figure}
Such rectangles are uniquely determined by their horizontal edges and therefore they are in bijection with pairs $(i,j)$ with $i < j$. In order to obtain a non-trivial closed cycle, one needs (at least) three rectangles, for instance, $(i,j),(j,\ell)$ and $(i,\ell)$ for $i < j < \ell$. These give rise to a \emph{cylinder}: pictorially, this is the union of the blue and green triangles in \autoref{fig:Delta3}, after appropriate relabeling; since the simplices $\Delta_1,\Delta_2$ are contracted, this cylinder is in fact a closed cycle, giving rise to a cohomology class. 

A cohomology class arising from rectangles $(i,j),(j,\ell),(i,\ell)$ for $i<j < \ell$ does not depend on $j$ and it coincides with the class obtained from $(i,j'),(j',\ell),(i,\ell)$ for $i < j'< \ell$. To see this, note that the rectangle with vertices $x_{1j'}, x_{1j}, x_{2j'},x_{2j}$ contracts in homology to a line which identifies edges $\overline{x_{1j'} x_{2j'}}$ and $\overline{x_{1j} x_{2j}}$. 

These classes generate $\rmH^2(\Delta)$. To see this, as noted above if a diagonal edge appears in a cohomology class, since the class is closed, the edge must appear in exactly 2 triangles, which form a rectangle. The left and right edges of such a rectangle are homologous to points since $\Delta$ contains the simplices $\Delta_1$ and $\Delta_2$, so any homology class is obtained by gluing a number of such objects along the horizontal edges. Hence generators of $\rmH^2(\Delta)$ are in bijection with pairs $(i,\ell)$ with $i < \ell-1$; write $[i,\ell]$ for the class corresponding to the pair $(i,\ell)$. 

Given a squarefree multidegree $\sigma$, identified with a subset of variables, the cohomology $\rmH^2(\Delta|_\sigma)$ is generated by (restrictions of) classes $[i,\ell]$ with the property that there exists $j$ with $i < j < \ell$ such that $\{x_{1i},x_{1j},x_{1\ell},x_{2i},x_{2j},x_{2\ell}\} \subseteq \sigma$. To see this, first notice that the $1$-simplices $\{ x_{1p},x_{2q} \}$ with $p < q$ do not affect $\rmH^2(\Delta |_\sigma)$ if $x_{2p},x_{1q} \notin \sigma$ since there are no $2$-simplices in $\Delta |_\sigma$ whose boundaries are those $1$-simplices. Let $\Delta'$ be the subcomplex of $\Delta |_{\sigma}$ obtained after deleting those $1$-simplices. In fact, $\Delta'$ is homotopically equivalent to $\Delta |_{\sigma'}$, where 
\[
\sigma' = \{ x_{1i},x_{2i} : i = 1 \vvirg n \text{ such that } x_{1i},x_{2i} \in \sigma  \}.
\]
Hence $\rmH^2(\Delta |_\sigma)$ only depends on the horizontal edges in $\Delta |_\sigma$.

In order to determine the number of such classes, define the \emph{weight} of a multidegree $\sigma$ to be 
\[
w(\sigma) = \#\{ i : \{ x_{1i} , x_{2i}\} \subseteq \sigma\}.
\]
Note that $\dim \rmH^2(\Delta|_\sigma)$ only depends on $w(\sigma)$, and we have $\dim \rmH^2(\Delta|_\sigma) = \binom{w(\sigma)-1}{2}$. By summing over all possible weights and all choices of $\sigma$ with a given weight, we obtain
\[
b_{k,k+3} = \sum_{\sigma: |\sigma| = k+3} \dim \rmH^2(\Delta|_{\sigma}) = \sum_{w = 0}^{\lfloor\frac{k+3}{2}\rfloor} \sum_{\sigma : \begin{smallmatrix}|\sigma|=k+3, \\ w(\sigma) = w\end{smallmatrix}} \binom{w-1}{2}.
\]
The number of multidegrees $\sigma$ with $|\sigma|=k+3, w(\sigma) = w$ is $2^{k+3-2w} \binom{n}{w} \binom{n-w}{k+3-2w} $. Indeed there are $\binom{n}{w}$ choices of the indices $i$ such that $\{ x_{1i},x_{2i}\} \subseteq \sigma$; then there are $\binom{n-w}{k+3-2w}$ choices for the indices $j$ such that exactly one among $x_{1j},x_{2j}$ belongs to $\sigma$; finally for each index $j$ there is a choice between whether $x_{1j}$ or $x_{2j}$ belongs to $\sigma$ and this accounts for the factors $2$.

Since the inner summation only depends on $w$, and it is $0$ for $w = 0,1,2$, we conclude
\[
b_{k,k+3} = \sum_{w = 3}^{\lfloor\frac{k+3}{2}\rfloor} 2^{k+3-2w} \binom{n}{w} \binom{n-w}{k+3-2w} \binom{w-1}{2}.
\]
\end{proof}

\subsection{From the third to the second row of the Betti table}

Computing the value of the Betti numbers in the second row of the Betti table of $R/P_{2 \times n}$ directly seems highly non-trivial. In this section, we use a Hilbert series argument, combined with the results of \cite{elsw} and \autoref{CMreg}, to obtain an expression for the \emph{differences of diagonal Betti numbers} 
\[
b_{k,k+2} - b_{k-1,k+2}, 
\]
see \autoref{hilbSeries}. The first result is that the generators of $P_{2 \times n}$ have no linear syzygies, \autoref{rmk: first row}. Then, in \autoref{sec: spectral sequence}, we will show that the third row of the Betti table of $R/\init(P_{2 \times n})$ computed in \autoref{lemma:thirdRowInitial} coincides with the third row of the Betti table of $R/P_{2 \times n}$. So \autoref{hilbSeries} allows us to compute the second row of the Betti table of $R/P_{2 \times n}$. 
\vskip .2in

\begin{lem}\label{rmk: first row}
The generators of $P_{2 \times n}$ have no linear syzygies. 
\end{lem}
\begin{proof}
Since $P_{2 \times n}$ is generated in degree $2$, the number of its linear syzygies is given by the difference $2n \cdot \mathrm{HF}_{P_{2 \times n}}(2) - \mathrm{HF}_{P_{2 \times n}}(3)$, where $\mathrm{HF}_{P_{2 \times n}}$ denotes the Hilbert function of the ideal. We show that this difference is $0$, or equivalently $\mathrm{HF}_{P_{2 \times n}}(3) = 2n \cdot \mathrm{HF}_{P_{2 \times n}}(2)$. To see this, note that $P_{2 \times n}$ and $\init(P_{2 \times n})$ have the same Hilbert function, therefore $\mathrm{HF}_{P_{2 \times n}}(3)$ can be computed directly using the presentation of $\init(P_{2 \times n})$ obtained from \autoref{thm: gb permanents}. The degree $3$ monomials of $\init(P_{2 \times n})$ are as follows:
\begin{itemize}
    \item for $i < j < k$, we have $x_{1i} x_{1j} x_{2k}$ and $x_{1i} x_{2j} x_{2k}$ directly from the cubic generators; these are $2 \binom{n}{3}$ monomials;
    \item for $i > j > k$, we have $x_{1i} x_{1j} x_{2k}$ and $x_{2i} x_{1j} x_{2k}$, generated by $x_{1j} x_{2k}$, and $x_{1i} x_{2j} x_{2k}$ and $x_{1i} x_{2j} x_{1k}$, generated by $x_{1i}x_{2j}$; these are $4 \binom{n}{3}$ monomials;
    \item for $i > j$, we have $x_{1i} x_{1j} x_{2j}$ and $x_{1i}x_{2i}x_{2j}$, generated by $x_{1i}x_{2j}$; these are $2 \binom{n}{2}$ monomials. 
\end{itemize}
We obtain
\[
\mathrm{HF}_{P_{2 \times n}}(3) = \textstyle \mathrm{HF}_{\init(P_{2 \times n})}(3) = 2 \binom{n}{3} + 4 \binom{n}{3} + 2 \binom{n}{2} = 2n \binom{n}{2} = 2n \cdot \mathrm{HF}_{P_{2 \times n}}(2),
\]
as desired.
\end{proof}

\autoref{rmk: first row} and \autoref{CMreg} guarantee that the Betti numbers of $P_{2 \times n}$ satisfy $b_{i,j}=0$ for all $i,j >1$, except when $j=i+2,i+3$. The following fact will allow us to determine the Hilbert series of $R/P_{2 \times n}$ using the $f$-vector computed in \autoref{lemma:f vector Delta}. In \autoref{hilbSeries}, we will use this result to compute the differences $b_{i,i+2} - b_{i-1,i+2}$. Let $\mathrm{HS}_M(t) = \sum (\dim M_k)t^k$ denote the Hilbert series of a graded module $M$ with homogeneous components $M_k$.

\begin{lem}[{\cite[Corollary 1.15]{MS}}]\label{lem:HF from f vector}
   Let $\Delta$ be a simplicial complex of dimension $N$ with $f$-vector $[f_{-1} \vvirg f_N]$. The Hilbert series of $R/ I_\Delta$ is given by 
   \[
   \mathrm{HS}_{R/I_{\Delta}} ( t) = t^{N+1}\frac{f(\Delta, \frac{1 - t}{t})}{(1-t)^{N+1}}.
   \]
\end{lem}
The vector of the coefficients of $t^{N+1}f(\Delta, \frac{1 - t}{t})$ is called the $h$-vector of the simplicial complex. We refer to \cite{Z} for the combinatorial relations between the $f$-vector and the $h$-vector and the resulting properties of the simplicial complex.

\begin{cor}\label{hilbSeries}
The Betti numbers of $P_{2 \times n}$ satisfy 
\[
b_{k,k+2} - b_{k-1,k+2} = 2 \binom{n}{k+2} - \binom{2n}{k+2} + \binom{n+1}{2} \binom{2n-2}{k} - 2 \binom{n}{2}\binom{2n-3}{k-1} 
\]
with the convention that $\binom{a}{b} = 0 $ if $b > a$ or if $b < 0$.
\end{cor}
\begin{proof}
The rings $R/ P_{2 \times n}$ and $R/ \init(P_{2 \times n})$ have the same Hilbert series. 
The exactness of the free resolution guarantees $\HS_{R/P_{2 \times n}} (t)= \sum_{k\geq 0} (-1)^{k} \HS_{F_k}(t)$, where $F_k$ is the $k$-th free module in the resolution of $R/P_{2 \times n}$. We have $F_0 = R$, and $F_1 = R(-2)^{b_{1,2}}$, with $b_{1,2} = \binom{n}{2}$. By \autoref{CMreg} and \autoref{rmk: first row}, there are only two nonzero Betti numbers $b_{k,k+2},b_{k,k+3}$ for every $k \geq 2$. 

\noindent Therefore $F_k = R( -(k+2))^{b_{k,k+2}} \oplus R( -(k+3))^{b_{k,k+3}}$; in particular 
\[\HS_{F_k} (t) = \frac{1}{(1-t)^{2n}} \cdot (b_{k,k+2} t^{k+2} + b_{k,k+3} t^{k+3}).\] Hence the Hilbert series $\HS_{R/P_{2 \times n}} (t)$ may be written as 
\[
\begin{array}{ccc}
\sum_{k\geq 0} (-1)^{k} \HS_{F_k}(t) &= & \frac{1}{(1-t)^{2n}} - \frac{b_{1,2} t^2}{(1-t)^{2n}} + \sum_{k\geq 2} \frac{(-1)^{k}}{(1-t)^{2n}}  \cdot (b_{k,k+2} t^{k+2} + b_{k,k+3} t^{k+3})  \\ 
&=& \frac{1}{(1-t)^{2n}} - \frac{b_{1,2} t^2}{(1-t)^{2n}} + \sum_{k\geq 2} \frac{(-1)^{k}}{(1-t)^{2n}}  \cdot (b_{k,k+2} - b_{k-1,k+2})t^{k+2}.
\end{array}
\]
We apply \autoref{lemma:f vector Delta} to obtain another expression for $\HS_{R/P_{2 \times n}}$ and compare the two expressions to write a formula for the differences of the Betti numbers. We have
\begin{align*}
\HS_{R/P_{2 \times n}} (t)=\mbox{ } & t^n \frac{f(\Delta, \frac{1-t}{t})}{(1-t)^n}  \\
= \mbox{ } & \frac{t^n}{(1-t)^n} \Bigl[2 \left({\textstyle \frac{1-t}{t}} - 1\right)^n - ({\textstyle \frac{1-t}{t}})^n + \textstyle\binom{n+1}{2} ({\textstyle \frac{1-t}{t}})^{n-2} + 2 \textstyle\binom{n}{2} ({\textstyle \frac{1-t}{t}})^{n-3}\Bigr]   \\
 = \mbox{ } &\frac{1}{(1-t)^n} \Bigl[ 2  - (1-t)^n + \textstyle \binom{n+1}{2} t^2 (1-t)^{n-2} + 2 \textstyle \binom{n}{2} t^3 (1-t)^{n-3}\Bigr].
\end{align*}
The numerator of the expression above, multiplied by $(1-t)^n$, gives the generating function of the difference of the Betti numbers. More precisely, $(-1)^{k}(b_{k,k+2} - b_{k-1,k+2})$ is the coefficient of $t^{k+2}$ in the following expression:
\begin{align*}
   &(1-t)^n  \Bigl[ 2  - (1-t)^n + \textstyle \binom{n+1}{2} t^2 (1-t)^{n-2} + 2 \textstyle \binom{n}{2} t^3 (1-t)^{n-3}\Bigr]  \\
   &= \sum_{k = 0}^n (-1)^k 2 \textbinom{n}{k} t^k - \sum_{k=0}^{2n} (-1)^k \textbinom{2n}{k} t^k + \sum_{k=0}^{2n-2} (-1)^k \textbinom{n+1}{2}\textbinom{2n-2}{k} t^{k+2}\!+\!
    \sum_{k=0}^{2n-3} (-1)^k 2\textbinom{n}{2}\textbinom{2n-3}{k} t^{k+3}.
\end{align*}
Shift the summation indices appropriately to obtain the desired formula.
\end{proof}

\section{Bernstein-Gelfand-Gelfand and $\Tor_i(R/P_{2 \times n},\k)_{i+3}$}\label{sec: spectral sequence}

In order to complete the proof of \autoref{main}, we will show that the third row of the Betti table of $R/P_{2 \times n}$ coincides with the third row of the Betti table of $R/\init(P_{2\times n})$. This is the content of \autoref{thm: spectral sequence count}, which is the main result of this section. 

We obtain this result via the Bernstein-Gelfand-Gelfand (BGG) correspondence, which allows us to associate two complexes to the modules $R/P_{2 \times n}$ and $R/\init(P_{2 \times n})$. We then build a double complex; the result will be obtained by studying the resulting spectral sequence.

\subsection{The BGG correspondence.} \label{subsec:BGG}
Let $V = R_1$ be the linear span of the variables of the polynomial ring $R$, so that $R = \Sym V$ coincides with the symmetric algebra of $V$. Let $W = V^*$ be the dual space of $V$ and let $E = \Lambda W$ be the exterior algebra of $W$. We write $e_{11} \vvirg e_{2n} \in W$ for the basis dual to the basis $x_{11} \vvirg x_{2n}$ of $V$. For every finitely generated graded $R$-module $N = \bigoplus_{t \in \mathbb{Z}} N_t$, let $N^{\vee}$ denote its \emph{graded dual}, defined by
\[
\ N^{\vee} = \bigoplus_{t \in \bbZ} \Hom_{\mathbb{K}} (N_t, \mathbb{K}) = \bigoplus_{t \in \bbZ} N_t^{\vee}.
\]
The action of $R$ is naturally defined by $(s \smcdot \phi)(m) = \phi(s \smcdot m)$ for every $s \in R$, $\phi \in N^{\vee}$ and $m \in N$ homogeneous elements. 

The {\em Bernstein-Gelfand-Gelfand} correspondence associates to every finitely generated, graded $R$-module $N$ a complex of free $E$-modules $\tilde{\bf R}(N)$. The homological properties of the complex $\tilde{\bf R}(N)$ are strongly related to the ones of the module $N$. We refer to \cite[Section 7E]{Eisenbud:SyzygyBook} for a discussion on this subject. The complex of $E$-modules is defined as follows:
\begin{align*}
\tilde{\bR}(N) \colon \quad \ldots \to E \otimes N_k^{\vee} &\to E \otimes N_{k-1}^{\vee} \to \ldots \\
e \otimes \phi &\mapsto \sum_{i,j} e \cdot e_{i,j} \otimes x_{i,j} \cdot \phi.
\end{align*}
We use the convention that both $N^\vee$ and $E$ are graded positively, with $\deg(e) = s$ is $ e\in \Lambda^s W$ and $\deg(\phi) = k$ if $\phi \in N_k^\vee$. This is different from \cite[\S7B]{Eisenbud:SyzygyBook}, but this difference does not affect any of the results. We record here \cite[Proposition 7.21]{Eisenbud:SyzygyBook}, which relates the graded Betti numbers of $N$ to the homology of $\tilde{\bfR}(N)$.
\begin{prop}[{\cite[Proposition 7.21]{Eisenbud:SyzygyBook}}]\label{prop: BGG correspondence}
Let $N$ be a graded $R$-module, and let $\tilde{\bfR}(N)$  be the associated BGG complex of $E$-modules. Then, for every $k$ and $\ell$,
\[
\rmH_\ell(\tilde{\bR}(N))_{k+\ell} \simeq \Tor_k(N,\bbK)_{k+\ell}^{\vee}.
\]
\end{prop}
Following \autoref{prop: BGG correspondence}, our goal is to show that, for every $s$,
\[
\rmH_3(\tilde{\bR}(R/P_{2 \times n}))_{s+3} = \rmH_3(\tilde{\bR}(R/\init(P_{2 \times n}))_{s+3}.
\]
This is the statement that the third row of the Betti table of $R/P_{2 \times n}$ coincides with the one of $R/\init(P_{2 \times n})$, which was computed in \autoref{lemma:thirdRowInitial}.

We point out that to prove this result it suffices to consider the BGG complexes of $R/P_{2 \times n}$ and $R/\init(P_{2 \times n})$ in homological degrees $0 \vvirg 5$. These are the ones that control the third homology group. In the next section, we will construct a double complex which relates $\tilde{\bR}(R / \init(P_{2 \times n}))$ and $\tilde{\bR}(R / P_{2 \times n})$. By analyzing the resulting spectral sequence, we will obtain the desired result.

\subsection{A double complex associated to the BGG complexes}\label{subsec: deformation}

Since $(P_{2 \times n})_k$ and $(\init(P_{2 \times n}))_k$ have the same dimension for every $k$, we may define an isomorphism between the spaces $(R/P_{2 \times n})_k^*$ and $(R/\init(P_{2 \times n}))_k^*$, and identify them with a vector space $V_k$. In this way the complexes $\tilde{\bR}(R / P_{2 \times n})$ and $\tilde{\bR}(R / \init(P_{2 \times n}))$ have the \emph{same} free modules $V_k \otimes E$ and they are characterized by their differentials. 

Since we are only interested in the third homology group of the BGG complexes, we restrict our analysis to the subcomplexes of $\tilde{\bR}(R / P_{2 \times n})$ and $\tilde{\bR}(R / \init(P_{2 \times n}))$ in homological degree $0 \vvirg 5$. 

Moreover, we consider the canonical identification between the spaces $(R/P_{2 \times n})_k^*$ and the annihilators $(P_{2 \times n})_k^\perp \subseteq R_k^*$, and similarly for $\init(P_{2 \times n})$. We use variables $y_{ij}$ dual to $x_{ij} \in R$ so that $R_k^*$ is the space of homogeneous polynomials of degree $k$ in the $y_{ij}$'s. In this way, the differentials of the BGG complex coincide with the restriction of the (rescaled) Koszul differential on the homogeneous components of $E \otimes R^\vee$. Formally, $y_{ij} = e_{ij} \in W = R_1^*$, but we prefer to use a different notation to avoid confusion between the product structure of $R^\vee$ and the product structure of the exterior algebra $E$. 

For $k = 0 \vvirg 5$, let $T_k$ be the linear span of the monomials in the variables $y_{ij}$ in $(P_{2 \times n})^\perp_k$; in particular $T_k \subseteq (\init(P_{2 \times n}))_k^\perp$. The spaces $T_k$ can be computed combinatorially. For instance $T_2$ is spanned by monomials which do not occur in the support of any $2 \times 2$ subpermanent: these have the form $y_{1i}y_{2i}$, $y_{1i}y_{1j}$ and $y_{2i}y_{2j}$.

For $k = 0,1$, $(P_{2 \times n})_k = (\init(P_{2 \times n}))_k = 0$, hence the annihilator coincides with the entire dual space $R_k^*$. For $k=2\vvirg 5$, we fix bases of $(P_{2 \times n})^\perp_k$ and a corresponding basis of $(\init(P_{2 \times n}))^\perp$ given by the initial elements of the basis $(P_{2 \times n})^\perp_k$. This is done as follows.
\begin{itemize}
    \item Consider the basis of $(P_{2 \times n})_2$ consisting of the monomials of $T_2$ and the following elements, for $i<j$: 
    \[
     y_{1i}y_{2j} - y_{2i}y_{1j}.
    \]
    The corresponding basis of $\init(P_{2 \times n})_2$ consists of the monomials of $T_2$ and the initial monomials $y_{1i}y_{2j}$ of the elements above. 
    \item Consider the basis of $(P_{2 \times n})_3$ consisting of the monomials of $T_3$ and the following elements, for $i<j$:
\[
 \begin{array}{ccc}
   y_{1i}( y_{1i}y_{2j} - y_{2i}y_{1j} ), & & y_{2i}(y_{1i}y_{2j} - y_{2i}y_{1j}),\\
   y_{1j}(y_{1i}y_{2j} - y_{2i}y_{1j}), & &y_{2j}(y_{1i}y_{2j} - y_{2i}y_{1j}).
 \end{array}
\]
The corresponding basis of $\init(P_{2 \times n})_3$ consists of the monomials of $T_3$ and the initial monomials of the elements above, which are always the left-most terms. 

\item Consider the basis of $(P_{2 \times n})_4$ consisting of the monomials of $T_4$ and the following elements, for $i<j$:
\[
\begin{array}{ccc}
y_{1i}^2(y_{1i}y_{2j} - y_{2i}y_{1j}), & & y_{2i}^2(y_{1i}y_{2j} -  y_{2i}y_{1j}),\\ 
y_{1i}y_{1j}( y_{1i}y_{2j} -  y_{2i}y_{1j} ), & & 
   y_{2i}y_{2j}( y_{1i}y_{2j} -  y_{2i}y_{1j} ),  \\ 
y_{1j}^2(y_{1i}y_{2j}-y_{2i}y_{1j}), & & y_{2j}^2(y_{1i}y_{2j}-y_{2i}y_{1j}), \\
\multicolumn{3}{c}{y_{1i}y_{2i}( y_{1i}y_{2j} -  y_{2i}y_{1j} ),} \\
\multicolumn{3}{c}{y_{1j}y_{2j}( y_{1i}y_{2j} -  y_{2i}y_{1j} ),} \\ 
\multicolumn{3}{c}{y_{1i}^2 y_{2j}^2 -  y_{1j}y_{2i}(y_{1i}y_{2j} - y_{2i}y_{1j}).}
\end{array}
\]
The corresponding basis of $\init(P_{2 \times n})_4$ consists of the monomials of $T_4$ and the initial monomials of the elements above, which are always the left-most terms. 

\item  Consider the basis of $(P_{2 \times n})_5$ consisting of the monomials of $T_5$ and the following elements, for $i<j$:
\[
\begin{array}{ccc}
y_{1i}^3(y_{1i}y_{2j} -  y_{2i}y_{1j}), & & y_{2i}^3(y_{1i}y_{2j} -  y_{2i}y_{1j}),\\ 
y_{1i}^2y_{1j}( y_{1i}y_{2j} -  y_{2i}y_{1j} ), & & y_{2i}^2y_{2j}( y_{1i}y_{2j} -  y_{2i}y_{1j} ), \\
y_{1i}y_{1j}^2( y_{1i}y_{2j} -  y_{2i}y_{1j} ), & & y_{2i}y_{2j}^2( y_{1i}y_{2j} -  y_{2i}y_{1j} ), \\
y_{1j}^3(y_{1i}y_{2j} -  y_{2i}y_{1j}), & & y_{2j}^3(y_{1i}y_{2j} -  y_{2i}y_{1j}), \\
 y_{1i}(y_{1i}^2 y_{2j}^2 -  y_{1j}y_{2i}(y_{1i}y_{2j} - y_{2i}y_{1j})), & & y_{2i}(y_{1i}^2 y_{2j}^2 -  y_{1j}y_{2i}(y_{1i}y_{2j} - y_{2i}y_{1j})),\\ 
 y_{1j}(y_{1i}^2 y_{2j}^2 -  y_{1j}y_{2i}(y_{1i}y_{2j} - y_{2i}y_{1j})), & & y_{2j}(y_{1i}^2 y_{2j}^2 -  y_{1j}y_{2i}(y_{1i}y_{2j} - y_{2i}y_{1j})), \\
 \multicolumn{3}{c}{y_{1i}^2y_{2i}( y_{1i}y_{2j} -  y_{2i}y_{1j} ),} \\ 
\multicolumn{3}{c}{y_{1j}^2y_{2j}( y_{1i}y_{2j} -  y_{2i}y_{1j} ),} \\ 
\multicolumn{3}{c}{y_{1i}y_{2i}^2( y_{1i}y_{2j} -  y_{2i}y_{1j} ),} \\
\multicolumn{3}{c}{y_{1j}y_{2j}^2( y_{1i}y_{2j} -  y_{2i}y_{1j} ).} \\
\end{array}
\]
the corresponding basis of $\init(P_{2 \times n})_5$ consists of the monomials of $T_5$ and the initial monomials of the elements above, which are always the left-most terms.
\end{itemize}
Let $V_0 \vvirg V_5$ be vector spaces with $\dim V_k = \dim (P_{2 \times n})^\perp_k$. Let $\iota_k : (P_{2 \times n})^\perp_k \to V_k$ and $\iota^{\init}_k : \init(P_{2 \times n})^\perp_k \to V_k$ be isomorphisms, mapping the basis elements described above, and their initial monomials, to the same basis elements of $V_k$. 

With these definitions, the subcomplexes of the BGG complexes of $R/P_{2 \times n}$ and $R/\init(P_{2 \times n})$ obtained restricting to homological degrees $0 \vvirg 5$ can be regarded as complexes 
\[
\cdots \to E \otimes V_k \to E \otimes V_{k-1} \to \cdots 
\]
with differentials $\delta^{\full}_k = \iota_{k-1} \circ \delta_k \circ \iota^{-1}_k$ and $\delta^{\init}_k = \iota^{\init}_{k-1} \circ \delta_k \circ {\iota^{\init}_k}^{-1}$, respectively. Here $\delta_k$ is the (dual) Koszul differential.

We have the following fundamental result:
\begin{lem}\label{lemma: commuting diagram}
 For $k = 0 \vvirg 4$, we have
 \[
 \delta^{\init}_k \circ \delta^{\full}_{k+1} = - \delta^{\full}_k \circ \delta^{\init}_{k+1} 
 \]
\end{lem}
\begin{proof}
 This can be checked directly on the basis elements. The result is trivial for $k = 0 ,1$ because in those cases $\delta^{\full}_k = \delta^{\init}_{k}$ hence both compositions are identically $0$.
 
In the case $k=2$, we sketch the proof for the element $v \in V_3$ defined by
\begin{align*}
 \iota_k (y_{1i}^2y_{2j} - y_{1i}y_{2i}y_{1j}) = \iota^{\init}_k(y_{1i}^2y_{2j}).
\end{align*}
We need to show 
 \[
 \delta^{\init}_k \circ \delta^{\full}_{k+1} (v) =  - \delta^{\full}_k \circ \delta^{\init}_{k+1}(v)
 \]
 or equivalently
 \[
 \delta_2 \circ {\iota^{\init}_2}^{-1} \circ \iota _2 \circ \delta_{3} (y_{1i}^2y_{2j} - y_{1i}y_{2i}y_{1j}) = -
\delta_2 \circ {\iota_2}^{-1} \circ \iota^{\init}_2 \circ \delta_3 ( y_{1i}^2y_{2j} )
 \]
We have 
 \begin{align*}
 \delta_{3} (y_{1i}^2y_{2j} - y_{1i}y_{2i}y_{1j}) =  &e_{1i} y_{1i}y_{2j} + e_{2j}y_{1i}^2 - e_{1i} y_{2i}y_{1j} - e_{2i}y_{1i}y_{1j} - e_{1j}y_{1i}y_{2i} = \\ 
 & e_{1i} (y_{1i}y_{2j} - y_{2i}y_{1j}) - e_{1j}y_{1i}y_{2i} - e_{2i}y_{1i}y_{1j} + e_{2j}y_{1i}^2 .
 \end{align*}
The isomorphism ${\iota_2^{\init}}^{-1} \circ \iota_2$ maps the first term to its leading term $ e_{1i} y_{1i}y_{2j}$, and the others to themselves because they are generated by elements of $T_2$. We conclude
\begin{align*}
\delta_2 \circ {\iota^{\init}_2}^{-1} \circ &\iota _2 \circ \delta_{3} (y_{1i}^2y_{2j} - y_{1i}y_{2i}y_{1j}) = \\ &\delta_2 ( e_{1i} y_{1i}y_{2j} - e_{1j}y_{1i}y_{2i} - e_{2i}y_{1i}y_{1j} + e_{2j}y_{1i}^2 ) =\\
& e_{1i} e_{2j} y_{1i} - e_{1j} e_{1i} y_{2i} - e_{1j} e_{2i} y_{1i} - e_{2i}e_{1i}y_{1j} - e_{2i}e_{1j}y_{1i} + e_{2j}e_{1i}y_{1i} = \\
&- e_{1j} e_{1i} y_{2i} - e_{2i}e_{1i}y_{1j} .
\end{align*}
Similarly, we have 
\begin{align*}
\delta_2 \circ {\iota_2}^{-1} &\circ \iota^{\init}_2 \circ \delta_3 ( y_{1i}^2y_{2j} ) = \delta_2 \circ {\iota_2}^{-1} \circ \iota^{\init}_2 ( e_{1i} y_{1i}y_{2j} + e_{2j} y_{1i}^2 ) = \\
&\delta_2 (  e_{1i} (y_{1i}y_{2j} - y_{1j}y_{2i}) + e_{2j} y_{1i}^2 ) = \\
&-  e_{1i} e_{1j} y_{2i} - e_{1i}e_{2i}y_{1j}.
\end{align*}
The proof in general can be performed easily using computer algebra software. In fact, we point out that since all basis elements are supported on (at most) five column indices of the $2 \times n$ matrix, it is enough to prove the statement for $n =5$.
\end{proof}

\autoref{lemma: commuting diagram} allows us to define a double complex $\bbF_{i,j} = E \otimes V_{i+j}$, with horizontal differentials given by $\delta^{\full}$ and vertical differentials given by $(-1)^j \delta^{\init}$:
{\tiny
\[
\xymatrix{
\ar[d]^{\delta^{\init}} & \ar[d]^{-\delta^{\init}} & \ar[d]^{\delta^{\init}} & \ar[d]^{-\delta^{\init}} & \ar[d]^{\delta^{\init}} & \ar[d]^{-\delta^{\init}} &   \\
E \otimes V_0 \ar[d]^{\delta^{\init}} & E \otimes V_1 \ar[d]^{-\delta^{\init}} \ar[l]_{\delta^{\full}}& E \otimes V_2 \ar[d]^{\delta^{\init}} \ar[l]_{\delta^{\full}}& E \otimes V_3 \ar[d]^{-\delta^{\init}} \ar[l]_{\delta^{\full}} & E \otimes V_4 \ar[d]^{\delta^{\init}} \ar[l]_{\delta^{\full}} & E \otimes V_5 \ar[d]^{-\delta^{\init}} \ar[l]_{\delta^{\full}} & 0 \ar[l]_{\delta^{\full}}  \ar[d]^{\delta^{\init}} \\
0 & E \otimes V_0 \ar[l]_{\delta^{\full}} \ar[d]^{-\delta^{\init}}& E \otimes V_1 \ar[l]_{\delta^{\full}} \ar[d]^{\delta^{\init}}& E \otimes V_2 \ar[l]_{\delta^{\full}} \ar[d]^{-\delta^{\init}} & E \otimes V_3 \ar[l]_{\delta^{\full}} \ar[d]^{\delta^{\init}} & E \otimes V_4 \ar[l]_{\delta^{\full}} \ar[d]^{-\delta^{\init}}& E \otimes V_5 \ar[l]_{\delta^{\full}} \ar[d]^{\delta^{\init}} & 0 \ar[d]^{-\delta^{\init}} \ar[l]_{\delta^{\full}} \\
& 0 & E \otimes V_0 \ar[l]_{\delta^{\full}} \ar[d]^{\delta^{\init}}& E \otimes V_1\ar[d]^{-\delta^{\init}} \ar[l]_{\delta^{\full}} & E \otimes V_2 \ar[l]_{\delta^{\full}} \ar[d]^{\delta^{\init}}& E \otimes V_3 \ar[l]_{\delta^{\full}} \ar[d]^{-\delta^{\init}} & E \otimes V_4 \ar[l]_{\delta^{\full}} \ar[d]^{\delta^{\init}}& E \otimes V_5 \ar[l]_{\delta^{\full}} \ar[d]^{-\delta^{\init}} \\
& & & & & & & }
\]
}
Every row and every column defines a complex. Moreover, \autoref{lemma: commuting diagram} guarantees that the diagram commutes. 

\subsection{The spectral sequence of the double complex}
In this section, we use a spectral sequence argument to complete the proof. A complete exposition of spectral sequences is beyond the scope of this paper; for detailed expositions see Appendix A.3 of \cite{Ecommalg} or \S 9.3 of \cite{SSS}. We briefly describe the construction.

Fix an integer $N$ and truncate the double complex $\bbF$ to a double complex $\bbD$ with $\bbD_{i,j} = 0$ if $i < -N$ or $j < -N$ and $\bbD_{i,j} = \bbF_{i,j}$ otherwise. One associate to $\bbD$ two collections of sequences of complexes, usually called the spectral sequences of $\bbD$ according to the vertical or the horizontal filtration.

Since every row and every column of $\bbD$ is a complex, one can consider the homology resulting from the vertical differentials $\delta^{\init}$ and obtain a collection of complexes whose terms are homology modules $E^1_{(i,j)}$ arising from the vertical differentials and whose maps are induced by the horizontal differentials $\delta^{\full}$. This is the \emph{first page} of the spectral sequence of $\bbD$ according to the vertical filtration. Taking the homology with respect to the horizontal differentials $\delta^{\full}$ gives a collection of modules $E^2_{(i,j)}$ and the theory guarantees that these modules are connected by certain maps: for every $(i,j)$, there is a map $E^2_{(i,j)} \to E^2_{(i+1,j-2)}$ each sequence of such maps forms a complex. This is the second page of the spectral sequence. This process continues: for every $p$, the $p$-th page of the spectral sequence according to the vertical filtration is a collection of modules and maps $E^p_{(i,j)} \to E^p_{(i+p-1, j-p)}$. One can see that for every $d$, there exists $p$ such that all maps in and out of $E^p_{(i,j)}$ for $i+j = d$ are $0$; in this case one says that the $(i,j)$-th term of the spectral sequence degenerates and writes $E^{\infty}_{(i,j)}$ for the corresponding module. We observe that for every $p$ and every $k$, one has $\dim (E^p_{(i,j)})_k \leq \dim (E^{p-1}_{(i,j)})_k$ because $(E^p_{(i,j)})_k$ is a quotient of a subspace of $(E^{p-1}_{(i,j)})_k$. 

The same construction can be done by considering homology according to the horizontal differentials to obtain the first page and analogous maps for the higher pages. This is the spectral sequence of $\bbD$ with respect to the horizontal filtration, and we denote the corresponding modules by $\bar{E}^{p}_{(i,j)}$.

We record the following fundamental result, which is a consequence of \cite[Thm. 9.3.5, Thm. 9.3.6]{SSS}.
\begin{thm}\label{thm: 936 SSS}
For every $k$ and every $d$
\begin{equation}\label{eqn: vertical equals horizontal}
\sum_{i+j = d} \dim (E^{\infty}_{(i,j)})_k = \sum_{i+j = d} \dim (\bar{E}^{\infty}_{(i,j)})_k .
\end{equation}
\end{thm}

The following result shows that the spectral sequence defined by the vertical filtration, in homological degree $3$, degenerates almost entirely in the first page.

\begin{lem}\label{lem: first page vertical}
For every $(i,j)$ with $i+j = 3$, $i,j \geq -N$, except $(i,j) = (-N,N+3)$, we have $E^\infty_{(i,j)} = E^1_{(i,j)} = \rmH_3(\tilde{\bR}(R / \init(P_{2 \times n})))$.
\end{lem}
\begin{proof}
Consider the first page of the spectral sequence defined by the vertical filtration in homological degree $3$. For every $(i,j)$ with $i < -N$ or $j < -N$, we immediately have $E^{1}_{(i,j)} = 0$. 

For every $(i,j)$ with $i \geq -N$, $E^1_{(i,j)}$ is defined by the differentials $\delta^{\init}$, hence $E^1_{(i,j)} = \rmH_{i+j}(\tilde{\bR}(R / \init(P_{2 \times n})))$ except for $i = -N$, where the vertical differential of $\bbD$ is $0$. We are going to show that if $i+j = 3$, the maps 
\[
E^1_{(i,j-1)} \leftarrow E^1_{(i,j)} \leftarrow E^1_{(i,j+1)} 
\]
are identically $0$.

By \autoref{CMreg} the regularity of $R/\init(P_{2 \times n})$ is $3$. Hence \autoref{prop: BGG correspondence} guarantees $E^1_{(i,j+1)} = \rmH_4(\tilde{\bR}(R / \init(P_{2 \times n}))) = 0$. In particular, the map $E^1_{(i,j)} \leftarrow E^1_{(i,j+1)}$ is $0$, because its source is $0$.

To conclude, we show that the map 
\[
\delta^{\full} : \rmH_3(\tilde{\bR}(R / \init(P_{2 \times n})))  \to \rmH_2(\tilde{\bR}(R / \init(P_{2 \times n})))
\]
is zero. We check this on an explicitly given set of generators of $\rmH_3(\tilde{\bR}(R / \init(P_{2 \times n}))) $ as an $E$-module. First, note that the (classes modulo $\im(\delta^{\init})$ of) the following elements $\phi_{ijk}\in \init(P_{2 \times n})_3^\perp \otimes E$ for $i < j < k$ are a set of generators of $\rmH_3(\tilde{\bR}(R / \init(P_{2 \times n})))$:
    \[
    \begin{array}{ccl}
    \phi_{ijk} & =&y_{1i}y_{1j}y_{1k} e_{2i}e_{2j}e_{2k} \\
    &- &y_{1i}y_{1j}y_{2j} e_{1k}e_{2i}e_{2k} \\
    &+&y_{1i}y_{1k}y_{2k} e_{1j} e_{2i} e_{2j} \\
    &+& y_{1i} y_{2i} y_{2j} e_{1j} e_{1k} e_{2k} \\
    &+& y_{1i} y_{2i} y_{2k} e_{1j} e_{1k} e_{2j} \\
    &+ &y_{1j}y_{1k}y_{2k} e_{1i} e_{2i} e_{2j} \\
    &- &y_{1j} y_{2j} y_{2k} e_{1i} e_{1k} e_{2i}\\
    &+ &y_{2i}y_{2j}y_{2k} e_{1i} e_{1j} e_{1k} .
    \end{array}
    \]
This relies on the two isomorphisms of \autoref{Hochster} and \autoref{prop: BGG correspondence}. They give an isomorphism
\[
\rmH_3(\tilde{\bR}(R / \init(P_{2 \times n})))_{3+d} = \bigoplus_{|\sigma| = d+3} \rmH^{|\sigma|-3-1}(\Delta|_\sigma)
\]
where $\Delta$ is the Stanley-Reisner complex of $\init(P_{2 \times n})$. 

It turns out that as an $E$-module, $\rmH_3(\tilde{\bR}(R / \init(P_{2 \times n})))$ is generated by its degree $6$ component, corresponding to $\bigoplus_{|\sigma| = 6} \rmH^{2}(\Delta|_\sigma)$. This follows, as does the isomorphism of \autoref{Hochster}, from the proof in \S 3 of \cite{Reisner}. If $|\sigma|=6$, then $\rmH^{2}(\Delta|_\sigma) = 0$ unless $\sigma = \{x_{1i},x_{1j},x_{1k},x_{2i},x_{2j},x_{2k}\}$ with $i < j < k$; in this case $\rmH^{2}(\Delta|_\sigma)$ has a unique generator, combination of the $8$ triangles of $\Delta|_\sigma$. The corresponding class in $\rmH_3(\tilde{\bR}(R / \init(P_{2 \times n})))_{6}$ is expressed as an element in $\init(P_{2 \times n})_3^\perp \otimes E_3$ by $\phi_{ijk}$, where each monomial corresponds to one of the triangles in the generator of $\rmH^{2}(\Delta|_\sigma)$: for instance, the triangle $\{ x_{1i},x_{1j},x_{1k} \}$ is associated to the monomial $y_{1i}y_{1j}y_{1k} e_{2i}e_{2j}e_{2k}$. A similar element in $\rmH^2(\Delta|_{\sigma'})$ where $\sigma' = \{ x_{1i},x_{1j},x_{1k},x_{2i},x_{2j},x_{2k},x_{1l} \}$ is given by $\phi_{ijk} e_{1l}$.

We use the isomorphism induced by $\iota_{3}^{-1} \circ {\iota_3^{\init}}$ to express the elements $\phi_{ijk}$ as (classes of) elements of $(P_{2 \times n})_3^\perp \otimes E$:
\[
    \begin{array}{ccl}
    \psi_{ijk} &= &y_{1i}y_{1j}y_{1k}e_{2i}e_{2j}e_{2k} \\
    &- &y_{1j}(y_{1i}y_{2j} - y_{2i}y_{1j}) e_{1k}e_{2i}e_{2k}\\
    &+ &y_{1k}(y_{1i}y_{2k} - y_{2i}y_{1k}) e_{1j} e_{2i} e_{2j}\\
    &+ &y_{2i}(y_{1i}y_{2j} - y_{2i}y_{1j}) e_{1j} e_{1k} e_{2k} \\
    &+ &y_{2i}(y_{1i}y_{2k} - y_{2i}y_{1k}) e_{1j} e_{1k} e_{2j} \\
    &+ &y_{1k}(y_{1j}y_{2k} - y_{2j}y_{1k}) e_{1i} e_{2i} e_{2j}\\
    &- &y_{2j}(y_{1j}y_{2k} - y_{2j}y_{1k}) e_{1i} e_{1k} e_{2i} \\
    &+ &y_{2i}y_{2j}y_{2k} e_{1i} e_{1j} e_{1k} .
    \end{array}
\]
It suffices to show that the Koszul differential vanishes on $\psi_{ijk}$ to conclude that the map induced by $\delta^{\full}$ on $ \rmH_3(\tilde{\bR}(R / \init(P_{2 \times n})))$ is identically $0$. This is done via an explicit calculation.
 \end{proof}

\begin{thm}\label{thm: spectral sequence count}
For every $k$
\[
\dim \mathrm{H}_3(\tilde{\bR}(R / \init(P_{2 \times n})))_k = \dim \mathrm{H}_3(\tilde{\bR}(R / P_{2 \times n}))_k.
\]
In particular, the third row of the Betti table of $R / \init(P_{2 \times n})$ coincides with the third row of the Betti table of $R / P_{2 \times n}$.
\end{thm}
\begin{proof}
Let $H(k)$ be the quantity in the statement of \autoref{thm: 936 SSS}. For $i = -N+1 \vvirg N+3$, \autoref{lem: first page vertical} shows $E^{\infty}_{(i,j)} = \rmH_3(\tilde{\bR}(R / \init(P_{2 \times n})))$. Therefore,
\[
H(k) \geq (2N+3) \dim \rmH_3(\tilde{\bR}(R / \init(P_{2 \times n}))_k.
\]
The same calculation as \autoref{lem: first page vertical}, performed on the horizontal filtration, shows $\bar{E}^1_{(i,j)} = \rmH_3(\tilde{\bR}(R / P_{2 \times n}))$ for $j = -N+1 \vvirg N+3$ and $0$ if $j < -N$ or $j > N+3$. For every $k$, we have $\dim (\bar{E}^\infty_{(i,j)})_k \leq \dim (\bar{E}^1_{(i,j)})_k$, which is bounded above by a quantity $C(k)$ not depending on $N$; therefore
\[
H(k) \leq (2N+3) \dim \rmH_3(\tilde{\bR}(R / P_{2 \times n}))_k + C(k).
\]
We deduce
\[
0 = H(k) - H(k) \geq (2N+3) ( \dim \rmH_3(\tilde{\bR}(R / \init(P_{2 \times n})))_k - \dim \rmH_3(\tilde{\bR}(R / P_{2 \times n}))_k) - C(k).
\]
This holds for every $N$, hence 
\[
\dim \rmH_3(\tilde{\bR}(R / \init(P_{2 \times n})))_k - \dim \rmH_3(\tilde{\bR}(R / P_{2 \times n}))_k \leq 0.
\]
By \autoref{prop: BGG correspondence}, we deduce 
\[
\dim \Tor_k(R / \init(P_{2 \times n}),\bbK)_{k+3} \leq \dim \Tor_k(R / P_{2 \times n},\bbK)_{k+3}.
\]
On the other hand, the Betti numbers of $R / P_{2 \times n}$ are bounded from above by the Betti numbers of $R / \init(P_{2 \times n})$, hence equality holds.
\end{proof}

\section{Representation Theory Aspects}
In this section, we show how our result yields insight into the representation theory related to permanents. The polynomial ring $\bbC[x_{ij} : i = 1,2, j=1 \vvirg n]$ has a natural action of the two symmetric groups $\frakS_2$ and $\frakS_n$ acting on the first and second index of the variables $x_{ij}$: if $(\sigma, \tau) \in \frakS_{2} \times \frakS_n$ then $(\sigma,\tau) \cdot x_{ij} = x_{\sigma_i,\tau_j}$. The homogeneous components of the ideal $P_{2 \times n}$ are invariant under this action. As a consequence, the resolution of $P_{2 \times n}$ has a $(\frakS_2 \times \frakS_n)$-equivariant structure. This structure is not easy to understand in general; we record here some results.

A partition $\lambda$ of an integer $n$ is a non-increasing integer sequence of length (at most) $n$, $\lambda_1 \geq \lambda_2 \geq \cdots \geq \lambda_n \geq 0$ with $\sum_1^k \lambda _i = n$. We identify two partitions if they are equal up to trailing zeros. 

For a partition $\lambda$ of $n$, let $[\lambda]$ denote the associated Specht module; this is an irreducible representation for $\frakS_n$ and a basis is given by \emph{standard Young tableaux} of shape $\lambda$ and content $\{1 \vvirg n\}$. The Specht modules are all the irreducible representations of $\frakS_n$ up to isomorphism. We refer to \cite{Ful:YT} for the details of the construction.

For every $n$, let $V_n$ be a representation of $\frakS_n$. We say that the collection $V_\bullet$ is \emph{inherited} from the $k$-th level if, for $n \geq k$, $V_n$ is a quotient of the induced representation $\Ind_{\frakS_{k} \times \frakS_{n-k}} ( V_k \otimes [n-k])$. For instance, the standard Weyl action of $\frakS_n$ on $\bbC^n$, given by permuting the elements of a fixed basis, is inherited from the $1$-st level: in fact $\bbC^n = \Ind_{\frakS_{1} \times \frakS_{n-1}}^{\frakS_n} ([1] \times [n-1])$. 

\subsection{The ideal $P_{2 \times n}$} Under the action of $\frakS_n$, we immediately have that the homogeneous component $[P_{2 \times n}]_2$ is isomorphic, as a $\frakS_n$-representation, to the space $\bbC^{\binom{[n]}{2}}$ with basis indexed by $2$-subsets of $[n]$: the basis element $e_{ij}$ is identified with the permanent $p_{ij} = x_{1i}x_{2j} + x_{1j}x_{2i}$.

Note that for every $2$-subset $\{ i,j\} \subseteq [n]$, the permanent $p_{ij}$ is invariant for the action of the copy of $\frakS_{2} \subseteq \frakS_{n}$ acting by swapping the indices $i$ and $j$. More precisely, we have 
\[
[P_{2 \times n}]_2 = \Ind^{\frakS_n}_{\frakS_2 \times \frakS_{n-2}} ( [2] \times [n-2]).
\]
Therefore, the action of $\frakS_n$ is inherited from the $2$-nd level.

We have an equivariant form of the first differential of the resolution. Let $Z_1(n) = \bbC^{\binom{[n]}{2}}$ with a basis $e_{ij}$; define
\[
\phi_1 : Z_1(n)  \to [P_{2 \times n}]_2
\]
mapping  $e_{ij}$ to $p_{ij}$; tensoring $\phi_1$ by $R$ gives a map of $R$-modules which coincide with the first differential of the resolution. One can compute the decomposition of $[P_{2 \times n}]_2$ in terms of Specht modules, for instance via the Littlewood-Richardson rule \cite[Sec. 7.3]{Ful:YT}:
\[
 [P_{2 \times n}]_2 = [n] \oplus [n-1,1] \oplus [n-2,2].
\]

\subsection{The first syzygy} From \autoref{main}, the kernel of the map $\phi_1$ is generated by a subspace $Z_2(n) \subseteq Z_1(n) \otimes R_2$ of dimension 
\[
a_2 = \binom{\binom{n}{2}}{2} + 3 \binom{n}{4} + 2 \cdot 2 \binom{n}{4}.
\]
We will show that $Z_2(n)$ is inherited from the $4$-th level. Write $e_{ij}$ for a basis of the space $Z_1(n)$ as before.

First, we consider Koszul syzygies, which arise in two ways, inherited respectively from $3$-rd and $4$-th level:
\begin{itemize}
    \item Consider the Koszul syzygy among subpermanents sharing one column:
    \[
    K_{3} = \langle e_{ij} \otimes p_{ik} - e_{ik} \otimes p_{ij} : i,j,k \text{ distinct indices} \rangle \subseteq Z_1(n)  \otimes R_2.
    \]
For $n = 3$, we have $Z_2(3) = K_3 = [2,1]$ as a $\frakS_3$-representation. In general $K_3 = \Ind_{\frakS_{3} \times \frakS_{n-3}}^{\frakS_4} Z_2(3)$.

\item Consider the Koszul syzygies among \emph{disjoint} subpermanents:
    \[
    K_{4} = \langle e_{ij} \otimes p_{k\ell} - e_{k\ell} \otimes p_{ij} : i,j,k,\ell \text{ distinct indices} \rangle \subseteq Z_1(n)  \otimes R_2.
    \]
We have $K_4 = \Ind_{\frakS_{4} \times \frakS_{n-4}}^{\frakS_4} Z_2(4)$. One can determine the irreducible decomposition of $K_4$, but it is not very enlightening.
\end{itemize}
There are other two families of syzygies. 

Let $m_{ij} = x_{1i}x_{2j} - x_{1j}x_{2i}$ be the $2 \times 2$ minors corresponding to columns $i,j$. Then 
\[
M = \langle \sum_{\substack{ \sigma \in \frakS_4 \\ \sigma_1 = 1}}(-1)^{ \sigma} e_{i_{\sigma_1} i_{\sigma_2}} \otimes m_{i_{\sigma_3}i_{\sigma_4}} :  i_1 \vvirg i_4 \text{ distinct} \rangle \subseteq   Z_1(n)  \otimes R_2,
\]
where $(-1)^{\sigma}$ denotes the sign of the permutation $\sigma \in \frakS_4$. One can verify $M = [2,1,1]$ when $n = 4$ and in general $M = \Ind^{\frakS_n}_{\frakS_{4} \times \frakS_{n-4}} [2,1,1]$.

Finally, define
\[
N = \langle  \sum _{\substack{ \tau \in \frakS_{4} \\ \tau \{0,1\} = \{0,1\}}}
(-1)^\tau x_{\eps i_{\tau_1}} x_{\eps i_{\tau_3}} \otimes e_{i_{\tau_2}i_{\tau_4}} : \eps = 0,1, \text{ and } i_1 \vvirg i_4 \text{ distinct} \rangle \subseteq  Z_1(n)  \otimes R_2.
\]
For $n=4$, one can verify $N = [2,2]$ as a $\frakS_4$-representation. In general $N = \Ind^{\frakS_n}_{\frakS_4 \times \frakS_{n-4}}[2,2]$. 

\subsection{Directions for Future Research}
A detailed analysis as in the case of the first syzygy is much more involved for the higher syzygies: additional difficulties arise when the syzygies are mixed between linear and quadratic, which is the case starting from the second syzygy. We pose this as a direction for future research.
\vskip .1in
\noindent{\bf Acknowledgments} 
Much of this work took place during a visit to Santa Fe Institute for a working group on Geometric Complexity Theory, and we thank SFI for providing a wonderful and stimulating environment, and J.M. Landsberg for initial conversations.

{\small 
\bibliographystyle{alphaurl}
\bibliography{perm.bib}

\begin{thebibliography}{BLMW11}

\bibitem[AT92]{AT}
N.~Alon and M.~Tarsi.
\newblock {Colorings and orientations of graphs}.
\newblock {\em Combinatorica}, 12(2):125–134, 1992.
\newblock \href {https://doi.org/10.1007/BF01204715}
  {\path{doi:10.1007/BF01204715}}.

\bibitem[BCP99]{BCP}
D.~Bayer, H.~Charalambous, and S.~Popescu.
\newblock {Extremal Betti Numbers and Applications to Monomial Ideals}.
\newblock {\em J. Algebra}, 2(221):497–512, 1999.
\newblock \href {https://doi.org/10.1006/jabr.1999.7970}
  {\path{doi:10.1006/jabr.1999.7970}}.

\bibitem[BGL13]{BGL}
J.~Buczynski, A.~Ginensky, and J.~M. Landsberg.
\newblock {Determinantal equations for secant varieties and the
  {E}isenbud–{K}oh–{S}tillman conjecture}.
\newblock {\em J. Lond. Math. Soc.}, 88(1):1–24, 2013.

\bibitem[BIP19]{BuIkPa:no_occurrence_obstructions_in_GCT}
P.~Bürgisser, C.~Ikenmeyer, and G.~Panova.
\newblock {No occurrence obstructions in geometric complexity theory}.
\newblock {\em J. Amer. Math. Soc.}, 32(1):163–193, 2019.
\newblock \href {https://doi.org/10.1090/jams/908}
  {\path{doi:10.1090/jams/908}}.

\bibitem[BLMW11]{BLMW}
P.~Bürgisser, J.~M. Landsberg, L.~Manivel, and J.~Weyman.
\newblock {An overview of mathematical issues arising in the {G}eometric
  {C}omplexity {T}heory approach to {$\mathit{VP} \neq \mathit{VNP}$}}.
\newblock {\em SIAM J. Comput.}, 40(4):1179–1209, 2011.
\newblock \href {https://doi.org/10.1137/090765328}
  {\path{doi:10.1137/090765328}}.

\bibitem[BT20]{BoTe:WaringRankSyz}
M.~Boij and Z.~Teitler.
\newblock {A bound for the Waring rank of the determinant via syzygies}.
\newblock {\em Lin. Alg. Appl.}, 587:195–214, 2020.
\newblock \href {https://doi.org/10.1016/j.laa.2019.11.007}
  {\path{doi:10.1016/j.laa.2019.11.007}}.

\bibitem[EH00]{E}
D.~Eisenbud and J.~Harris.
\newblock {\em {The geometry of schemes}}, volume 197 of {\em {Graduate Texts
  in Mathematics}}.
\newblock Springer-Verlag, New York, 2000.

\bibitem[Eis88]{e1}
D.~Eisenbud.
\newblock {Linear sections of determinantal varieties}.
\newblock {\em Amer. J. Math.}, 110(3):541–575, 1988.
\newblock \href {https://doi.org/10.2307/2374622} {\path{doi:10.2307/2374622}}.

\bibitem[Eis95]{Ecommalg}
D.~Eisenbud.
\newblock {\em Commutative algebra}, volume 150 of {\em Graduate Texts in
  Mathematics}.
\newblock Springer-Verlag, New York, 1995.
\newblock With a view toward algebraic geometry.

\bibitem[Eis05]{Eisenbud:SyzygyBook}
D.~Eisenbud.
\newblock {\em {The geometry of syzygies}}, volume 229 of {\em {Graduate Texts
  in Mathematics}}.
\newblock Springer-Verlag, New York, 2005.

\bibitem[ELSW18]{elsw}
K.~Efremenko, J.~M. Landsberg, H.~Schenck, and J.~Weyman.
\newblock {On minimal free resolutions of sub-permanents and other ideals
  arising in complexity theory}.
\newblock {\em J. Algebra}, 503:8–20, 2018.
\newblock \href {https://doi.org/10.1016/j.jalgebra.2018.01.021}
  {\path{doi:10.1016/j.jalgebra.2018.01.021}}.

\bibitem[EN62]{EN}
J.~A. Eagon and D.~G. Northcott.
\newblock {Ideals defined by matrices and a certain complex associated with
  them}.
\newblock {\em Proc. Roy. Soc. London Ser. A}, 269(1337):188--204, 1962.

\bibitem[ER98]{ER}
J.~A. Eagon and V.~Reiner.
\newblock {Resolutions of {S}tanley-{R}eisner rings and {A}lexander duality}.
\newblock {\em J. Pure and Appl. Alg.}, 130(3):265–275, 1998.
\newblock \href {https://doi.org/10.1016/S0022-4049(97)00097-2}
  {\path{doi:10.1016/S0022-4049(97)00097-2}}.

\bibitem[Ful97]{Ful:YT}
W.~Fulton.
\newblock {\em {Young tableaux. With applications to representation theory and
  geometry}}, volume~35 of {\em {London Mathematical Society Student Texts}}.
\newblock Cambridge University Press, Cambridge, 1997.

\bibitem[GIP17]{GesIkPa:GCTMatrixPowering}
F.~Gesmundo, C.~Ikenmeyer, and G.~Panova.
\newblock {Geometric complexity theory and matrix powering}.
\newblock {\em Diff. Geom. Appl.}, 55:106–127, 2017.
\newblock \href {https://doi.org/10.1016/j.difgeo.2017.07.001}
  {\path{doi:10.1016/j.difgeo.2017.07.001}}.

\bibitem[GKKS13]{GKKS:ArithmeticCircuitsChasmDepthThree}
A.~Gupta, P.~Kamath, N.~Kayal, and R.~Saptharishi.
\newblock {Arithmetic circuits: A chasm at depth three}.
\newblock {\em Electronic Colloquium on Computational Complexity (ECCC)},
  20:26, 2013.
\newblock \href {https://doi.org/10.1109/FOCS.2013.68}
  {\path{doi:10.1109/FOCS.2013.68}}.

\bibitem[GL19]{GesLan:ExplicitPolysMaxPartialsMulmuley}
F.~Gesmundo and J.~M. Landsberg.
\newblock {Explicit polynomial sequences with maximal spaces of partial
  derivatives and a question of {K. Mulmuley}}.
\newblock {\em Theory of Computing}, 15(3):1 – 24, 2019.
\newblock \href {https://doi.org/10.4086/toc.2019.v015a003}
  {\path{doi:10.4086/toc.2019.v015a003}}.

\bibitem[Hoc72]{H}
M.~Hochster.
\newblock {Rings of invariants of tori, Cohen-Macaulay rings generated by
  monomials,and polytopes. \hskip -.1in}.
\newblock {\em Ann. Math.}, 96:318–337, 1972.
\newblock \href {https://doi.org/10.2307/1970791} {\path{doi:10.2307/1970791}}.

\bibitem[IL17]{IkLan:Compl_of_perm_in_various_comp_models}
C.~Ikenmeyer and J.~M. Landsberg.
\newblock {On the complexity of the permanent in various computational models}.
\newblock {\em J. Pure Appl. Algebra}, 221(12):2911–2927, 2017.
\newblock \href {https://doi.org/10.1016/j.jpaa.2017.02.008}
  {\path{doi:10.1016/j.jpaa.2017.02.008}}.

\bibitem[IP17]{IkPa:Rectangular_Kron_in_GCT}
C.~Ikenmeyer and G.~Panova.
\newblock {Rectangular Kronecker coefficients and plethysms in geometric
  complexity theory}.
\newblock {\em Adv. Math.}, 319:40–66, 2017.
\newblock \href {https://doi.org/10.1016/j.aim.2017.08.024}
  {\path{doi:10.1016/j.aim.2017.08.024}}.

\bibitem[Kir08]{K}
G.~A. Kirkup.
\newblock {Minimal primes over permanental ideals}.
\newblock {\em Trans. Amer. Math. Soc.}, 360(7):3751–3770, 2008.
\newblock \href {https://doi.org/10.1090/S0002-9947-08-04340-7}
  {\path{doi:10.1090/S0002-9947-08-04340-7}}.

\bibitem[Lan17]{Lan:GeometryComplThBook}
J.~M. Landsberg.
\newblock {\em {Geometry and complexity theory}}, volume 169 of {\em {Cambridge
  Studies in Advanced Mathematics}}.
\newblock Cambridge University Press, Cambridge, 2017.

\bibitem[Las78]{L}
A.~Lascoux.
\newblock {Syzygies des variétés déterminantales}.
\newblock {\em Adv. Math.}, 30(3):202–237, 1978.
\newblock \href {https://doi.org/10.1016/0001-8708(78)90037-3}
  {\path{doi:10.1016/0001-8708(78)90037-3}}.

\bibitem[LS00]{LS}
R.~C. Laubenbacher and I.~Swanson.
\newblock {Permanental ideals}.
\newblock {\em J. Symb. Comp.}, 30(2):195–205, 2000.
\newblock \href {https://doi.org/10.1006/jsco.2000.0363}
  {\path{doi:10.1006/jsco.2000.0363}}.

\bibitem[MNS12]{Mul3}
K.~D. Mulmuley, H.~Narayanan, and M.~Sohoni.
\newblock {Geometric complexity theory III: on deciding nonvanishing of a
  Littlewood–Richardson coefficient}.
\newblock {\em J. Alg. Combinatorics}, 36(1):103–110, 2012.
\newblock \href {https://doi.org/10.1007/s10801-011-0325-1}
  {\path{doi:10.1007/s10801-011-0325-1}}.

\bibitem[MS01]{Mul1}
K.~D. Mulmuley and M.~Sohoni.
\newblock {Geometric {C}omplexity {T}heory. {I}. {A}n approach to the {P} vs.\
  {NP} and related problems}.
\newblock {\em SIAM J. Comput.}, 31(2):496–526 (electronic), 2001.
\newblock \href {https://doi.org/10.1137/S009753970038715X}
  {\path{doi:10.1137/S009753970038715X}}.

\bibitem[MS05]{MS}
E.~Miller and B.~Sturmfels.
\newblock {\em {Combinatorial Commutative Algebra}}, volume 227 of {\em
  {Graduate Texts in Mathematics}}.
\newblock Springer, New York, 2005.

\bibitem[MS08]{Mul2}
K.~D. Mulmuley and M.~Sohoni.
\newblock {Geometric {C}omplexity {T}heory. {II}. {T}owards explicit
  obstructions for embeddings among class varieties}.
\newblock {\em SIAM J. Comput.}, 38(3):1175–1206, 2008.
\newblock \href {https://doi.org/10.1137/080718115}
  {\path{doi:10.1137/080718115}}.

\bibitem[Rei76]{Reisner}
G.~A. Reisner.
\newblock Cohen-{M}acaulay quotients of polynomial rings.
\newblock {\em Advances in Math.}, 21(1):30--49, 1976.

\bibitem[Sch03]{S}
H.~Schenck.
\newblock {\em {Computational algebraic geometry}}, volume~58 of {\em {London
  Mathematical Society Student Texts}}.
\newblock Cambridge University Press, Cambridge, 2003.

\bibitem[Sch22]{SSS}
H.~Schenck.
\newblock {\em Algebraic foundations for applied topology and data analysis},
  volume~1 of {\em Mathematics of Data}.
\newblock Springer, Cham, springer, 2022.
\newblock URL: \url{https://doi.org/10.1007/978-3-031-06664-1}.

\bibitem[SS12]{SidS}
J.~Sidman and G.~G. Smith.
\newblock {Linear determinantal equations for all projective schemes}.
\newblock {\em Algebra Number Theory}, 5(8):1041--1061, 2012.

\bibitem[Val79]{V}
L.~G. Valiant.
\newblock {Completeness classes in algebra}.
\newblock In {\em {Proceedings of the 11th ACM Symp. on Th. of Comp.}}, {STOC
  '79}, page 249–261, New York, 1979. ACM.
\newblock \href {https://doi.org/10.1145/800135.804419}
  {\path{doi:10.1145/800135.804419}}.

\bibitem[Zie95]{Z}
G.~M. Ziegler.
\newblock {\em {Lectures on polytopes}}, volume 152 of {\em {Graduate Texts in
  Mathematics}}.
\newblock Springer-Verlag, New York, 1995.

\end{thebibliography}
}
\end{document}